\documentclass[11pt]{amsart}

\makeindex

\usepackage{amsfonts,graphics,amsmath,amsthm,amsfonts,amscd,amssymb,amsmath,latexsym,multicol,euscript}
\usepackage{epsfig,url}
\usepackage{flafter}
\usepackage{hyperref}
\usepackage[all,cmtip,line]{xy}
\usepackage{array}
\usepackage[english]{babel}
\usepackage{overpic}
\usepackage{subfig}
\usepackage{multirow}

\usepackage{pgf,tikz}
\usepackage{mathrsfs}
\usetikzlibrary{arrows}

\usepackage{wrapfig}

\usepackage{enumitem}

\theoremstyle{definition}

\theoremstyle{theorem}
\newtheorem{theoremalpha}{Theorem}
\theoremstyle{corollary}

\newtheorem{theorem}{Theorem}[section]
\newtheorem{main theorem}{Main Theorem}
\newtheorem{proposition}[theorem]{Proposition}

\newtheorem{corollary}[theorem]{Corollary}

\newtheorem{lemma}[theorem]{Lemma}

\newtheorem{theorem*}{Theorem}
\newtheorem{corollary*}[theorem*]{Corollary}
\newtheorem{conjecture*}[theorem*]{Conjecture}

\theoremstyle{definition}
\newtheorem{example}[theorem]{Example}
\newtheorem{definition}[theorem]{Definition}
\newtheorem{definition-lemma}[theorem]{Definition-Lemma}

\newtheorem{remark}[theorem]{Remark}

\newcommand\R{\mathbb{R}}
\newcommand\Q{\mathbb{Q}}
\newcommand\Z{\mathbb{Z}}
\newcommand\C{\mathbb{C}}
\renewcommand\P{\mathbb{P}}

\newcommand\eps{\varepsilon}
\renewcommand\epsilon{\varepsilon}

\newcommand{\mc}{\mathcal}

\newcommand{\ol}[1]{\overline{#1}}

\DeclareMathOperator{\N^1}{N^1}

\DeclareMathOperator{\ord}{ord}
\DeclareMathOperator{\mult}{mult}

\DeclareMathOperator{\Supp}{Supp}

\DeclareMathOperator{\vol}{vol}
\DeclareMathOperator{\val}{val}
\DeclareMathOperator{\divisor}{div}

\DeclareMathOperator{\Eff}{Eff}

\DeclareMathOperator{\Bigdiv}{Big}

\newcommand{\bm}{\mathbf B_-}  
\newcommand{\B}{\mathbf B}
\newcommand{\bp}{\mathbf B_+}  
\newcommand{\okbd}{\Delta}
\newcommand{\okval}{\Delta^{\val}}
\newcommand{\oklim}{\Delta^{\lim}}

\newcommand{\kappanu}{\kappa_\nu}

\begin{document}

\title{Okounkov bodies associated to pseudoeffective divisors}

\author{Sung Rak Choi}
\address{Department of Mathematics, Yonsei University, Seoul, Korea}
\email{sungrakc@yonsei.ac.kr}

\author{Yoonsuk Hyun}
\address{Samsung Advanced Institute of Technology, Suwon, Korea}
\email{yoonsuk.hyun@samsung.com}

\author{Jinhyung Park}
\address{School of Mathematics, Korea Institute for Advanced Study, Seoul, Korea}
\email{parkjh13@kias.re.kr}

\author{Joonyeong Won}
\address{Center for Geometry and Physics, Institute for Basic Science, Pohang, Korea}
\email{leonwon@ibs.re.kr}

\date{\today}
\thanks{S. Choi and J. Park were partially supported by the NRF grant (NRF-2016R1C1B2011446). J. Won was partially supported by IBS-R003-D1, Institute for Basic Science in Korea. }

\begin{abstract}
An Okounkov body is a convex subset in Euclidean space associated to a big divisor on a smooth projective variety with respect to an admissible flag. In this paper, we introduce two convex bodies associated to pseudoeffective divisors, called the valuative Okounkov bodies and the limiting Okounkov bodies, and show that these convex bodies reflect the asymptotic properties of pseudoeffective divisors as in the case with big divisors. Our results extend  the works of Lazarsfeld-Musta\c{t}\u{a} and Kaveh-Khovanskii. For this purpose, we define and study special subvarieties, called the Nakayama subvarieties and the positive volume subvarieties, associated to pseudoeffective divisors.
\end{abstract}

\maketitle


\section{Introduction}
For a divisor $D$ on a smooth projective variety $X$ of dimension $n$, one can associate a convex body $\okbd_{Y_\bullet}(D)$, called the Okounkov body of $D$, in the Euclidean space $\R^n$ with respect to an admissible flag $Y_\bullet$.
After the pioneering works by Lazarsfeld-Musta\c{t}\u{a} (\cite{lm-nobody}) and Kaveh-Khovanskii (\cite{KK}), motivated by earlier works by Okounkov (\cite{O1}, \cite{O2}),
the Okounkov bodies $\okbd_{Y_\bullet}(D)$ have received a considerable amount of attention in a variety of flavors. It is believed that this convex body carries rich information of the asymptotic invariants of $D$.
For example, it was proven in \cite[Theorem A]{lm-nobody} that if $D$ is big, then the Euclidean volume of $\okbd_{Y_\bullet}(D)$ in $\R^n$ is equal to the volume $\vol_X(D)$ of a divisor $D$ up to the constant $n!$, that is, we have
$$
\vol_{\R^n}(\okbd_{Y_\bullet}(D)) = \frac{1}{n!}\vol_X(D).
$$
However, little is known about the Okounkov bodies for non-big pseudoeffective divisors.
One of the annoying phenomena is that for a pseudoeffective divisor $D$ that is not big, the associated convex body $\okbd_{Y_\bullet}(D)$ is not full dimensional in $\R^n$ so that its Euclidean volume in $\R^n$ is zero.
Nevertheless, it is still tempting to study the asymptotic properties of pseudoeffective divisors using the associated Okounkov bodies.

In this paper, we introduce and study two different convex bodies, the valuative Okounkov body $\okval_{Y_\bullet}(D)$ and the limiting Okounkov body $\oklim_{Y_\bullet}(D)$, associated to a pseudoeffective divisor $D$ with respect to an admissible flag $Y_\bullet$ (see Definition \ref{def-1.1}).
Using these convex bodies, we extend some of the previous works of \cite{lm-nobody} and \cite{KK} on big divisors to the pseudoeffective divisors.
We will readily see that $\okval_{Y_\bullet}(D)$ and $\oklim_{Y_\bullet}(D)$ with respect to a \emph{suitable choice of an admissible flag $Y_\bullet$} encode the asymptotic invariants of the divisor $D$ as in the case of big divisors.

\smallskip

Turning to the details, we recall the construction of the Okounkov body that is equivalent to the ones given by Lazarsfeld-Musta\c{t}\u{a} (\cite{lm-nobody}) and Kaveh-Khovanskii (\cite{KK}).
Let $D$ be an $\R$-divisor on a projective variety $X$ of dimension $n$.
We fix an \emph{admissible flag} $Y_\bullet$ on $X$, which is, by definition, a sequence of subvarieties
$$
Y_{\bullet}:~ X=Y_0\supseteq Y_1 \supseteq\cdots\supseteq Y_{n-1}\supseteq Y_n=\{ x \}
$$
such that each $Y_i$ is an $(n-i)$-dimensional irreducible subvariety of $X$ that is nonsingular at a point $x$.
After possibly replacing $X$ by an open subset, we may suppose that each $Y_{i+1}$ is a Cartier divisor on $Y_i$.
Let $|D|_{\R}:=\{ D' \mid D \sim_{\R} D' \geq 0 \}$ and suppose that it is nonempty.
We define a valuation-like function
$$
\nu_{Y_\bullet} : |D|_{\R} \to \R^n_{\geq 0}
$$
as follows. For $D' \in |D_{\R}|$, let
$$
\nu_1=\nu_1(D'):=\ord_{Y_1}(D') ~~\text{ and }~~  \nu_2=\nu_2(D'):=\ord_{Y_2}((D'-\nu_1(D')Y_1)|_{Y_1}).
$$
Assuming we have defined $\nu_1,\nu_2,\cdots, \nu_{i-1}$, we inductively define $\nu_i$ as
$$
\nu_i= \nu_i(D') :=\ord_{Y_i}( ( \cdots ((D'-\nu_1 Y_1)|_{Y_1}-\nu_2 Y_2)|_{Y_2} - \cdots - \nu_{i-1}Y_{i-1})|_{Y_{i-1}}).
$$
By collecting the values $\nu_i = \nu_i(D')$, we obtain
$$
\nu_{Y_\bullet}(D') = (\nu_1, \cdots, \nu_n) \in \R^n_{\geq 0}.
$$
Following \cite{lm-nobody} and \cite{KK}, we define the \emph{Okounkov body} $\okbd_{Y_\bullet}(D)$ of a divisor $D$ with respect to an admissible flag $Y_\bullet$ as the closed set
$$
\okbd_{Y_\bullet}(D) := \text{the closure of the convex hull of $\nu_{Y_\bullet}(|D|_{\R})$ in $\R^n_{\geq 0}$}.
$$
In \cite{lm-nobody}, Lazarsfeld-Musta\c{t}\u{a} mainly consider the Okounkov body $\okbd_{Y_\bullet}(D)$ of a big divisor $D$.
However, the valuation-like function $\nu_{Y_\bullet}$ works for any divisors $D$ with nonempty $|D|_{\mathbb R}$, and the construction of the Okounkov body can be carried out for such divisors without any changes. See Subsection \ref{ok of LM} for more details.

Now we define the valuative Okounkov body and the limiting Okounkov body of a divisor.

\begin{definition}\label{def-1.1}
Let $D$ be an $\R$-divisor on a smooth projective variety $X$ of dimension $n$, and $Y_\bullet$ be an admissible flag on $X$.
\begin{enumerate}[leftmargin=0.3cm,itemindent=0.5cm]
 \item[$(1)$] Suppose that $D$ is effective up to $\R$-linear equivalence, that is, $|D|_{\R}\neq \emptyset$.
 The \emph{valuative Okounkov body} $\okval_{Y_\bullet}(D)$ of $D$ with respect to $Y_\bullet$ is defined as
$$
\okval_{Y_\bullet}(D) := \text{the closure of the convex hull of $\nu_{Y_\bullet}(|D|_{\R})$ in $\R^n_{\geq 0}$}.
$$
If $|D|_{\R} = \emptyset$, then we put $\okval_{Y_\bullet}(D) := \emptyset$.\footnote{
By definition, we have in fact $\okbd_{Y_\bullet}(D)=\okval_{Y_\bullet}(D)$ for any divisor $D$.
However, in this paper, we use the notation $\okbd_{Y_\bullet}(D)$ only when $D$ is a big divisor in order to stress the properties of $\okval_{Y_\bullet}(D)$ and $\oklim_{Y_\bullet}(D)$ for a non-big divisor $D$.
}
 \item[$(2)$] Suppose that $D$ is pseudoeffective. The \emph{limiting Okounkov body} $\oklim_{Y_\bullet}(D)$ of $D$ with respect to $Y_\bullet$ is defined as
$$
\oklim_{Y_\bullet}(D):=\lim_{\epsilon \to 0+} \okbd_{Y_\bullet}(D+\epsilon A) = \bigcap_{\epsilon >0} \okbd_{Y_\bullet}(D+\epsilon A) \text{ in $\R^n_{\geq 0}$}
$$
where $A$ is a fixed ample divisor on $X$.
If $D$ is not pseudoeffective, then we put $\oklim_{Y_\bullet}(D):=\emptyset$.
Note that the definition of the limiting Okounkov body $\oklim_{Y_\bullet}(D)$ is independent of the choice of the ample divisor $A$.
\end{enumerate}
\end{definition}

\smallskip

We briefly recall definitions of some of basic asymptotic invariants of divisors.
Let $D$ be an $\R$-divisor on a smooth projective variety $X$.
The \emph{restricted volume} of $D$ along a $v$-dimensional subvariety $V$ of $X$ is defined as
$$
\vol_{X|V}(D):=\limsup_{m \to \infty} \frac{h^0(X|V, \lfloor mD \rfloor)}{m^v/v!}
$$
where $h^0(X|V, \lfloor mD \rfloor)$ is the dimension of the image of the restriction map $H^0(X, \mathcal{O}_X(\lfloor mD \rfloor)) \to H^0(V, \mathcal{O}_V(\lfloor mD \rfloor))$. Note that $\vol_{X|X}(D)=:\vol_X(D)$ is the usual volume of the divisor $D$.
When $V \not\subseteq \bm(D)$, the \emph{augmented restricted volume} of $D$ along $V$ is defined as
$$
\vol_{X|V}^+(D):=\lim_{\epsilon \to 0+}\vol_{X|V}(D+\epsilon A)
$$
where $A$ is some fixed ample divisor on $X$. It is independent of the choice of the ample divisor $A$.
For more details, we refer to Subsection \ref{volsubsec}.
The \emph{Iitaka dimension} of $D$ is defined as
$$
\kappa(D):=\max \left\{ k \in \Z_{\geq 0} \left|\; \limsup_{m \to \infty} \frac{h^0(X, \mathcal{O}_X(\lfloor mD \rfloor))}{m^k}>0 \right.\right\}
$$
if  $h^0(X, \mathcal{O}_X(\lfloor mD \rfloor)) \neq 0$ for some $m > 0$, and $\kappa(D):=-\infty$ otherwise. Similarly, the \emph{numerical Iitaka dimension} of $D$ is defined as
$$
\kappanu(D):=\max \left\{ k \in \Z_{\geq 0} \left|\; \limsup_{m \to \infty} \frac{h^0(X, \mathcal{O}_X(\lfloor mD \rfloor + A))}{m^k}>0 \right.\right\}
$$
for some fixed ample Cartier divisor $A$ if $h^0(X, \mathcal{O}_X(\lfloor mD \rfloor+A)) \neq 0$ for some $m > 0$, and $\kappanu(D):=-\infty$ otherwise.  It is also independent of the choice of the ample divisor $A$. For more details, see Subsections \ref{subsec-Iitaka} and \ref{subsec-posvolloc}.

\smallskip

To extract asymptotic invariants of a divisor $D$ with $\kappa(D)\geq0$ from the valuative Okounkov body $\okval_{Y_\bullet}(D)$, we consider an admissible flag $Y_\bullet$ that contains a special subvariety $Y_{n-\kappa(D)}=U \subseteq X$ such that the natural restriction map $H^0(X, \mathcal{O}_X( \lfloor mD \rfloor )) \to H^0(U, \mathcal{O}_U(\lfloor mD|_U \rfloor))$ is injective for every integer $m \geq 0$. We call such $U$ a \emph{Nakayama subvariety} of $D$ (see Subsection \ref{subsec-Iitaka}).
Note that any general subvariety of dimension $\kappa(D)$ is a Nakayama subvariety of $D$ (Proposition \ref{prop-gen subvar=naka}).
If $D$ is big, then $X$ itself is the unique Nakayama subvariety of $D$.

\begin{theoremalpha}[=Theorem \ref{okval-main}]\label{thrm-val}
Let $D$ be an $\R$-divisor on a smooth projective variety $X$ of dimension $n$ such that $\kappa(D)\geq 0$.
Fix an admissible flag $Y_\bullet$ containing a Nakayama subvariety $U$ of $D$ such that $Y_n=\{x\}$ is a general point. Then $\okval_{Y_\bullet}(D) \subseteq \{0 \}^{n-\kappa(D)} \times \R^{\kappa(D)}_{\geq 0}$.
Furthermore, we have
$$
\dim \okval_{Y_\bullet} (D)=\kappa(D) \text{\;\; and\;\; } \vol_{\R^{\kappa(D)}}(\okval_{Y_\bullet}(D))=\frac{1}{\kappa(D)!} \vol_{X|U} (D)
$$
where $\okval_{Y_\bullet}(D)$ is regarded as a convex subset of $\R^{\kappa(D)}$.
\end{theoremalpha}

For convenience, we define $\dim (\text{point}) := 0$ and $\vol_{\R^0}(\text{point}):=1$ throughout the paper. Theorem \ref{thrm-val} does not hold if $D$ is only effective up to $\R$-linear equivalence (see Remark \ref{effuptoR}).

\smallskip

To investigate the asymptotic properties of a pseudoeffective divisor $D$ using the limiting Okounkov body $\oklim_{Y_\bullet}(D)$, we consider an admissible flag $Y_\bullet$ that contains a special subvariety $Y_{n-\kappanu(D)}=V \subseteq X$ such that  $\vol_{X|V}^+(D)>0$.
We call such $V$ a \emph{positive volume subvariety} of $D$ (see Subsection \ref{subsec-posvolloc}).
Note that that the intersection of $(n-\kappanu(D))$ general ample divisors is a positive volume subvariety of $D$ (Proposition \ref{general pos vol loc}).
Thus we may take a general admissible flag $Y_\bullet$ constructed by successively intersecting general ample divisors.
If $D$ is big, then $X$ itself is the unique positive volume subvariety of $D$.

\begin{theoremalpha}[=Corollary \ref{cor-oklimY}]\label{thrm-lim}
Let $D$ be a pseudoeffective $\R$-divisor on a smooth projective variety $X$ of dimension $n$.
Fix an admissible flag $Y_\bullet$ containing a positive volume subvariety $V$ of $D$. Then $\oklim_{Y_\bullet}(D) \subseteq \{0 \}^{n-\kappa_{\nu}(D)} \times \R^{\kappa_{\nu}(D)}_{\geq 0}$.
Furthermore, we have
$$
\dim \oklim_{Y_\bullet} (D)=\kappa_{\nu}(D) \text{\;\; and\;\; } \vol_{\R^{\kappa_{\nu}(D)}}(\oklim_{Y_\bullet} (D))=\frac{1}{\kappa_{\nu}(D)!} \vol^+_{X|V}(D)
$$
where $\oklim_{Y_\bullet}(D)$ is regarded as a convex subset of $\R^{\kappa_{\nu}(D)}$.
\end{theoremalpha}

In general, we have $\okval_{Y_\bullet}(D) \subseteq \oklim_{Y_\bullet}(D)$ (see Remark \ref{remk-restricted ok}), and the inclusion can be strict (see Example \ref{okval notsubset oklim}).
Note that if $D$ is big, then we have
$$
\okval_{Y_\bullet}(D)=\oklim_{Y_\bullet}(D)=\okbd_{Y_\bullet}(D).
$$
We also remark that Di Biagio-Pacienza \cite{DP} also studied valuative Okounkov bodies of effective divisors and the subvarieties on which the restricted divisors behave like big divisors.

Recall that the Okounkov body of a big divisor is a numerical invariant. More precisely, by \cite[Proposition 4.1]{lm-nobody} and \cite[Theorem A]{Jow},  for big divisors $D,D'$ on $X$, we have $\okbd_{Y_\bullet}(D) = \okbd_{Y_\bullet}(D')$ for any admissible flag $Y_\bullet$ if and only if $D \equiv D'$. As Theorem \ref{thrm-jow} states below, this result can be extended to the limiting Okounkov bodies $\oklim_{Y_\bullet}(D)$ for pseudoeffective divisors $D$. We note however that the valuative Okounkov body $\okval_{Y_\bullet} (D)$ is not a numerical invariant (see Remark \ref{ex-okval-not num}), but an invariant of $\R$-linear equivalence (see Proposition \ref{okvallin}).

\begin{theoremalpha}[=Theorem \ref{Jow for oklim}]\label{thrm-jow}
Let $D,D'$ be pseudoeffective $\R$-divisors on a smooth projective variety $X$.
Then
$$
\text{$D \equiv D'$ if and only if $\oklim_{Y_\bullet}(D)=\oklim_{Y_\bullet}(D')$ for every admissible flag $Y_\bullet$.}
$$
\end{theoremalpha}

Since the limiting Okounkov body is a numerical invariant, we can consider a function
$$
\oklim_{Y_\bullet} :  \overline{\text{Eff}}(X) \to \R^n_{\geq 0}, ~~\xi \mapsto \oklim_{Y_\bullet}(\xi)
$$
for a fixed admissible flag $Y_\bullet$ on a smooth projective variety $X$ of dimension $n$.
By \cite[Theorem B]{lm-nobody} and the definition of the limiting Okounkov body, the shapes of limiting Okounkov bodies continuously change. In contrast, Example \ref{okval notsubset oklim} also shows discontinuity of valuative Okounkov bodies at some non-big divisor class.

We finally remark that Boucksom also defined and studied the limiting Okounkov body (the \emph{numerical Okounkov body} in his terminology) in \cite{B}. See Remark \ref{remk-boucksom} for more details on his results. In this paper, we do not use any results in \cite{B}.

\medskip

This paper is organized as follows.
We start in Section \ref{prelimsec} by collecting relevant basic facts on asymptotic invariants such as asymptotic base locus, asymptotic valuation, (augmented) restricted volume, (numerical) Iitaka dimension, etc., and prove some useful properties of Nakayama subvarieties and positive volume subvarieties.
Section \ref{oksec} is the main part of this paper. We first recall the construction of the Okounkov body $\okbd_{Y_\bullet}(D)$ of a big divisor $D$ in Subsection \ref{ok of LM}.
We then introduce and study the valuative Okounkov body $\okval_{Y_\bullet}(D)$ and the limiting Okounkov body $\oklim_{Y_\bullet}(D)$ of a pseudoeffective divisor $D$ in Subsections \ref{okbdnaksubsec} and \ref{oklimsubsec}, respectively. Especially, we prove Theorems \ref{thrm-val}, \ref{thrm-lim}, and \ref{thrm-jow}. Finally, in Section \ref{exsec}, we exhibit various examples of the valuative Okounkov body $\okval_{Y_\bullet}(D)$ and the limiting Okounkov body $\oklim_{Y_\bullet}(D)$.

\subsection*{Aknowledgement}
We are grateful to Mihai Fulger for helpful suggestions and useful comments. We would like to thank the referee for careful reading of our manuscript and for making a number of valuable suggestions.

\section{Preliminaries}\label{prelimsec}

In this section, we collect basic facts and introduce some notions that will be used in later sections.
First, we fix some notations. Throughout the paper, we work over the field $\C$ of complex numbers.
For simplicity, a \emph{variety} in this paper is assumed to be smooth, projective, reduced and irreducible, but a \emph{subvariety} can be singular.
By a \emph{divisor} on a variety $X$, we always mean an $\R$-divisor unless otherwise stated. Since we assume that $X$ is smooth, every divisor is $\R$-Cartier.
A divisor $D$ is  called \emph{pseudoeffective} if its numerical class $[D]\in \N^1(X)_\R:=\N^1(X)\otimes \R$ lies in the pseudoeffective cone $\overline{\text{Eff}}(X)$, the closure of the cone $\N^1(X)_\R$ spanned by the classes of effective divisors. A divisor is called \emph{big} if its numerical class lies in the big cone $\text{Big}(X)$, the interior of $\ol\Eff(X)$.

\subsection{Asymptotic base locus}\hfill

Let $D$ be a divisor on a smooth projective variety $X$.
If $D$ is a $\Q$-divisor, then we define the \emph{stable base locus} of $D$ as
$$
\text{SB}(D):= \bigcap_m \text{Bs}(mD)
$$
where the intersection is taken over all positive integers $m$ such that $mD$ are $\Z$-divisors, and $\text{Bs}(mD)$ denotes the base locus of the linear system $|mD|$. 
Now we assume that $D$ is an $\R$-divisor.
The \emph{augmented base locus} $\bp(D)$ is defined as
$$
\bp(D):=\bigcap_A\text{SB}(D-A)
$$
where the intersection is taken over all ample divisors $A$ such that $D-A$ are $\Q$-divisors.
The \emph{restricted base locus} $\bm(D)$ of $D$ is defined as
$$
\bm(D):=\bigcup_{A}\text{SB}(D+A).
$$
where the union is taken over all ample divisors $A$ such that $D+A$ are $\Q$-divisors.
We have $\bm(D) \subseteq \text{SB}(D) \subseteq \bp(D)$ for a $\Q$-divisor $D$.
One can check that a divisor $D$ is ample (or nef) if and only if $\bp(D)=\emptyset$
(respectively, $\bm(D)=\emptyset$). Furthermore, $D$ is not pseudoeffective (or not big) if and only if $\bm(D)=X$ (respectively, $\bp(D)=X$).
The asymptotic base loci $\bp(D)$ and $\bm(D)$ depend only on the numerical class of $D$ while $\text{SB}(D)$ does not (see \cite[Example 10.3.3]{pos}). For more properties, we refer to \cite{elmnp-asymptotic inv of base}.

\subsection{Asymptotic valuation}\hfill

Let $D$ be a pseudoeffective divisor on a smooth projective variety $X$, and $V\subseteq X$ be an irreducible subvariety of $X$.
When $D$ is big, we define \emph{the asymptotic valuation} of $V$ at $D$ as
$$
\ord_V(||D||):=\inf\{\ord_V(D')\mid D\equiv D'\geq 0\}.
$$
When $D$ is non-big pseudoeffective, we define
$$
\ord_V(||D||):=\lim_{\epsilon\to 0+}\ord_V(||D+\epsilon A||)
$$
for some ample divisor $A$. This definition is independent of the choice of $A$, and the value $\ord_V(||D||)$ depends only on the numerical class $[D]\in\N^1(X)_\R$ (see \cite[Chapter III]{nakayama}).

\begin{remark}
By \cite[Proposition 6.4]{lehmann-red}, if $D$ is an abundant pseudoeffective divisor, then the limit process in computing $\ord_X(||D||)$ is unnecessary. Otherwise, it is necessary in general.
For example, let $\pi : S \to \P^2$ be the blow-up of $\P^2$ at nine general points. Then $|-K_S|$ consists of one irreducible curve $C$, and $|-mK_S|=\{ m C \}$ for any integer $m > 0$. Thus $\ord_C(D')=1$ for every effective divisor $D$ with $D \equiv -K_S$. On the other hand, since $-K_S$ is nef, $\ord_C(||-K_S||)=0$  by Theorem \ref{thrm-B- by sigmanum}.
\end{remark}

The restricted base locus $\bm(D)$ can be characterized in terms of $\ord_{V}(||D||)$ as follows.

\begin{theorem}[{\cite[Proposition 2.8]{elmnp-asymptotic inv of base},\cite[Lemma V.1.9]{nakayama}}]\label{thrm-B- by sigmanum}
Let $D$ be a pseudoeffective divisor on a variety $X$, and $V\subseteq X$ an irreducible subvariety. Then $V\subseteq\bm(D)$ if and only if $\ord_V(||D||)>0$.
\end{theorem}

The asymptotic valuation is a birational invariant: for a birational morphism $f:Y\to X$ with the exceptional divisor $E$ such that $f(E)=V$, we have $\ord_V(||D||)=\ord_E(||f^*(D)||)$ (see \cite[Lemma 1.4]{BBP}).  For more details on asymptotic valuations, see \cite{elmnp-asymptotic inv of base} and \cite[Chapter III]{nakayama}.

\subsection{Augmented restricted volume}\label{volsubsec}\hfill

In this subsection, we introduce the notion of the augmented restricted volume. For this purpose, we first recall the restricted volume function.
Let $D$ be a divisor on a smooth projective variety $X$ of dimension $n$, and $V$ be a $v$-dimensional irreducible subvariety of $X$.
For the natural restriction map $\varphi:H^0(X,\mc O_X(\lfloor D \rfloor))\to H^0(V,\mc O_V(\lfloor D \rfloor))$, we denote the image $\text{Im}(\varphi)$ by $H^0(X|V, \lfloor D \rfloor)$ and its dimension by $h^0(X|V,\lfloor D \rfloor )$. We define the \emph{restricted volume} of $D$ along $V$ as
$$
\vol_{X|V}(D):=\limsup_{m \to \infty} \frac{h^0(X|V,\mc O_X(\lfloor mD \rfloor) )}{m^{v}/v!}.
$$

\begin{remark}\label{restvol-rmrk}
In \cite{elmnp-restricted vol and base loci}, the restricted volume $\vol_{X|V}(D)$ is initially defined for a $\Q$-divisor $D$. If $V \not\subseteq \bp(D)$, then $\vol_{X|V}(D)$ depends only on the numerical class of a $\Q$-divisor $D$ (\cite[Corollary 2.14]{elmnp-restricted vol and base loci}). Furthermore, the restricted volume function uniquely extends to a continuous function
$$
\vol_{X|V} : \text{Big}^V (X) \to \R_{\geq 0}
$$
where $\text{Big}^V(X)$ is the set of all $\R$-divisor classes $\xi$ such that $V$ is not properly contained in any irreducible component of $\bp(\xi)$, and $\vol_{X|V}(\xi)=0$ if and only if $V$ is an irreducible component of $\bp(\xi)$ (\cite[Theorem 5.2]{elmnp-restricted vol and base loci}).

On the other hand, one can easily show that our restricted volume $\vol_{X|V}$ for $\Q$-divisors is homogeneous of degree $v$ so that our definition of $\vol_{X|V}(D)$ coincides with that in \cite{elmnp-restricted vol and base loci} for a $\Q$-divisor $D$. By perturbing by ample divisors and using the continuity property, one can also check that for any $\R$-divisor $D$ such that $[D] \in \text{Big}^V(X)$, our definition of $\vol_{X|V}(D)$ coincides with the value $\vol_{X|V}([D])$ of the continuous function $\vol_{X|V} : \text{Big}^V (X) \to \R_{\geq 0}$ in \cite[Theorem 5.2]{elmnp-restricted vol and base loci}. See also \cite[Remark 2.13]{lehmann-nu}.
\end{remark}

By Remark \ref{restvol-rmrk} and \cite[Corollary 2.14]{elmnp-restricted vol and base loci}, $\vol_{X|V}(D)$ depends only on the numerical class of an $\R$-divisor $D$, and by \cite[Theorem 5.2]{elmnp-restricted vol and base loci}, we obtain a continuous function
$$
\vol_{X|V} : \text{Big}^V (X) \to \R_{\geq 0},
$$
and $\vol_{X|V}(\xi)=0$ if and only if $V$ is an irreducible component of $\bp(\xi)$. Furthermore, by \cite[Corollary 2.15]{elmnp-restricted vol and base loci}, `$\limsup$' can be replaced by `$\lim$' in the definition of the restricted volume $\vol_{X|V}(D)$ of a divisor $D$ when $V \not\subseteq \bp(D)$.

By letting $V=X$, we recover the classical volume function $\vol_X(D)=\vol_{X|X}(D)$. Note that $ \text{Big}^X (X)=\text{Big}(X)$. Thus we obtain a continuous function $\vol_X : \text{Big}(X) \to \R_{>0}$. Furthermore, the volume function continuously extends to the entire N\'{e}ron-Severi space
$$
\vol_X : \N^1 (X)_{\R} \to \R_{\geq 0},
$$
and $\vol_X(\xi)=0$ if and only if $\xi \not\in \text{Big}(X)$ (see \cite[Corollary 2.2.45]{pos}).
For more details on the functions $\vol_X$ and $\vol_{X|V}$, see \cite{elmnp-restricted vol and base loci}, \cite{pos}, \cite{lehmann-nu}, etc.

If $D$ is a non-big pseudoeffective divisor, then $\bp(D)=X$.
Thus the functions $\vol_{X|V}$ and $\vol_X$ do not capture the subtle asymptotic properties of
pseudoeffective divisors that are not big on $V$ or $X$.
In such situations, the following function seems useful.

\begin{definition}\label{def-volume+}
Let $D$ be a pseudoeffective divisor on a smooth projective variety $X$,
and $V\subseteq X$ be an irreducible subvariety such that $V\not\subseteq\bm(D)$.
The \emph{augmented restricted volume} of $D$ along $V$ is defined as
$$
\vol_{X|V}^+(D):=\lim_{\eps\to 0+} \vol_{X|V}(D+\eps A)
$$
where $A$ is an ample divisor on $X$.
\end{definition}

Note that if $V\not\subseteq\bm(D)$, then $V \not\subseteq \bp(D+\eps A)$ for any ample divisor $A$ and for any $\eps > 0$.
Thus using the above basic properties of $\vol_{X|V}$, we can deduce that in Definition \ref{def-volume+}, the limit exists and $\vol^+_{X|V}(D)$ is independent of the choice of $A$.

We also observe that $\vol_{X|V}^+(D)$ depends only on the numerical class of $D$.
This can be seen as follows.
Let $D \equiv D'$ and $V \not\subseteq \bm(D)=\bm(D')$. Then $\vol_{X|V}(D+\eps A) = \vol_{X|V}(D' + \eps A)$ so that by taking the limit $\eps\to 0$, we obtain $\vol_{X|V}^+(D) = \vol_{X|V}^+(D')$.
By the continuity of the function $\vol_{X|V}$, we also see that
$$
\vol_{X|V}^+(D) = \vol_{X|V}(D) \text{ for $D\in\Bigdiv^V(X)$.}
$$
For a divisor $D$ such that $[D] \in  \overline{\text{Eff}}(X) \setminus \text{Big}(X)$ and $V \not\subseteq \bm(D)$, by definition, $\vol_{X|V}^+(D)$ is the limit of the values of the continuous function $\vol_{X|V}$. Thus we obtain a continuous function
$$
\vol_{X|V}^+ : \overline{\text{Eff}}^V(X) \to \R
$$
where $\overline{\text{Eff}}^V(X) := \text{Big}^V(X) \cup \{ \xi \in \overline{\text{Eff}}(X) \setminus \text{Big}(X) \mid V \not\subseteq \bm(\xi)   \}$.
For $D\in \overline{\text{Eff}}^V(X)$, the following inequalities hold by definition:
$$\vol_{X|V}(D) \leq \vol_{X|V}^+(D) \leq \vol_{V}(D|_V).$$
Both inequalities can be strict in general.

\begin{example}\label{vol<vol+}
Let $S$ be a relatively minimal rational elliptic surface, and $H$ be an ample $\Z$-divisor on $S$.
Take a general element $V \in |kH|$ for a sufficiently large integer $k >0$.
Then $\vol_{S|V}(-K_S)=\limsup\limits_{m \to \infty} \frac{h^0(S, \mathcal{O}_S(-mK_S))}{m}$ is apparently independent of $k$ and $H$.
However, one can check that $\vol_{S|V}^+(-K_S)= (-K_S) \cdot (kH)$. Thus  $\vol_{S|V}^+(-K_S)$ can be arbitrarily large depending on $k$. Thus we have $\vol_{S|V}(-K_S) < \vol_{S|V}^+(-K_S)$.
\end{example}

\subsection{Nakayama subvariety}\label{subsec-Iitaka}\hfill

In this subsection, we introduce and study Nakayama subvarieties of divisors, which are closely related to the Iitaka dimension. Let $X$ be a smooth projective variety.

\begin{definition}
For a divisor $D$ on a variety $X$, the \emph{Iitaka dimension} of $D$ is defined as
$$
\kappa(D):=\max \left\{ k \in \Z_{\geq 0} \left|\; \limsup_{m \to \infty} \frac{h^0(X, \mathcal{O}_X(\lfloor mD \rfloor))}{m^k}>0 \right.\right\}
$$
if  $h^0(X, \mathcal{O}_X(\lfloor mD \rfloor)) \neq 0$ for some $m > 0$, and $\kappa(D):=-\infty$ otherwise.
\end{definition}

\begin{remark}\label{remk-kappa=m^k}
If $D$ is a divisor such that $\kappa(D)\geq 0$, then there exists an integer $m_0\geq 1$ such that $h^0(X,\mc O_X(\lfloor mm_0D\rfloor))\sim m^{\kappa(D)}$ for $m\gg 0$ (\cite[Theorem II.3.7]{nakayama}).
For this $m_0$, we have $\kappa(D)=\kappa(m_0D)$.
Thus in many situations below, we may assume that $m_0=1$.
\end{remark}

\begin{remark}\label{remk-kappa not num}
Note that the Iitaka dimension $\kappa(D)$ is not a numerical invariant in general (\cite[Example 6.1]{lehmann-nu}). Moreover, it is not an $\R$-linear invariant in general, that is, $\kappa(D) \neq \kappa(D')$ even if $D \sim_{\R} D'$.
For instance, let $P,Q$ be distinct points on $\P^1$. For an irrational number $a \in \R$, we consider an $\R$-divisor $D:=a(P-Q)$. We have $D \sim_{\R} 0$, so $D$ is effective up to $\R$-linear equivalence.
But $\kappa(a(P-Q))=-\infty$ since $\lfloor ma(P-Q) \rfloor = \lfloor ma \rfloor P -(\lfloor ma \rfloor +1)Q < 0$ for any integer $m>0$.
However, if both $D, D'$ are effective and $D \sim_{\R} D'$, then $\kappa(D)=\kappa(D')$ holds true (see \cite[Remark 2.1]{G}).
\end{remark}

\begin{remark}\label{remk-on kappa}
We have the following implications:
$$
\text{$D$ is effective $\Rightarrow$ $\kappa(D) \geq 0$ $\Rightarrow$ $|D|_{\R} \neq 0$.}
$$
If the latter condition holds, then we say that $D$ is effective up to $\R$-linear equivalence.
By what we have seen above, it is easy to check that the converses are not true in general.
\end{remark}

The following is well known, but we include the whole proof for reader's convenience.

\begin{lemma}\label{r-principal}
Let $P$ be a principal $\R$-divisor. Then we may write $P = \sum_{i=1}^k a_i \divisor(f_i)$ where $a_1, \cdots, a_k$ are real numbers linearly independent over $\Q$ and $f_1, \cdots, f_k \in \C(X)$ are rational functions.
\end{lemma}

\begin{proof}
Let $k$ be the minimum of integers $m$ such that we may write $P=\sum_{i=1}^m a_i \text{div} (f_i)$ where $a_i$ are real numbers and $f_i \in \C(X)$ are rational functions.
We now fix an expression $P= \sum_{i=1}^k a_i \text{div}(f_i)$.
Suppose that $a_1, \cdots, a_k$ are not linearly independent over $\Q$. Possibly by reordering the indices, we have
$$
a_k = \frac{p_1}{q_1}  a_1 + \cdots + \frac{p_{k-1}}{q_{k-1}}  a_{k-1}
$$
where $p_i, q_i$ are relatively prime integers with $q_i \neq 0$.
Then we obtain
$$
P = \sum_{i=1}^{k-1} a_i \left( \text{div}(f_i) + \frac{p_i}{q_i} \text{div}(f_k) \right) = \sum_{i=1}^{k-1} \frac{a_i}{q_i} \text{div}\left( f_i^{q_i} f_k^{p_i}\right)
$$
which is a contradiction to the minimality of $k$.
\end{proof}

The following lemma is useful when we deal with  effective $\R$-divisors.

\begin{lemma}\label{r-divisor}
Let $D$ be an effective divisor on a variety $X$, and $D' \in |D|_{\R}$ be any element. Suppose that we can write
$D=\sum_i a_i E_i$ and $D'=\sum_j a_j' E_j'$ where $E_i, E_j'$ are  prime divisors and $a_i, a_j' > 0$.
Then for any sufficiently small $\eps > 0$, there exist a sufficiently large and divisible integer $m>0$ and an effective divisor $G \in |\lfloor mD \rfloor|$ such that we can write $\frac{G}{m} = \sum_i b_i E_i + \sum_j b_j' E_j'$ where  $ b_i < \eps$ and $|a_j' - b_j'| < \eps$ for all $i, j$.
\end{lemma}

\begin{proof}
Since $D \sim_{\R} D'$, there exist rational functions $f_k \in \C(X)$ and real numbers $c_k$ such that $D = D' + \sum_k c_k \text{div}(f_k)$. By Lemma \ref{r-principal}, we may assume that $c_k$ are linearly independent over $\Q$. Then $\Supp \text{div}(f_k)\subseteq\Supp(D) \cup \Supp(D')$ for each $k$.
There are arbitrarily small real numbers $c_k'$ such that $c_k - c_k'$ are rational numbers, $ D' +  \sum_k c_k'  \text{div}(f_k)$ is an effective divisor, and $|\mult_{E_j'}(D' +  \sum_k c_k'  \text{div}(f_k)) - a_j'|$ are arbitrarily small for all $j$. We now have
$$
D= \left( D' +  \sum_k c_k'  \text{div}(f_k) \right) + \sum_k (c_k - c_k')  \text{div}(f_k) \sim_{\Q} \left( D' +  \sum_k c_k'  \text{div}(f_k) \right).
$$
For a sufficiently large and divisible integer $m>0$, let $G:= \left\lfloor m \left(D' + \sum_k c_k' \text{div}(f_k) \right) \right\rfloor$. Then we have
$$
\lfloor mD \rfloor \sim G ~~\text{ and }~~ \frac{G}{m} = \frac{\lfloor \sum_j ma_j' E_j' + \sum_k mc_k' \text{div}(f_k) \rfloor}{m}.
$$
Since $m>0$ is sufficiently large, $\Supp(\text{div}(f_k)) \subseteq \Supp(D) \cup \Supp(D')$ and $c_k'$ are arbitrarily small for all $k$, we obtain the assertion.
\end{proof}

Now we define the Nakayama subvariety associated to an effective divisor.

\begin{definition}\label{def-nakayama locus}
Let $D$ be a divisor on a variety $X$ with $\kappa(D) \geq 0$.
An irreducible subvariety $U \subseteq X$  is called a \emph{Nakayama subvariety} of $D$ if $\kappa(D)=\dim U$ and the natural map
$$
H^0(X, \mc O_X(\lfloor mD \rfloor)) \to H^0(U, \mc O_U(\lfloor mD|_U \rfloor))
$$
is injective (or equivalently, $H^0(X, \mc I_U \otimes \mc O_X(\lfloor mD \rfloor))=0$ where $\mc I_U$ is an ideal sheaf of $U$ in $X$) for every integer $m \geq 0$.
\end{definition}

By definition, $U=X$ is the (unique) Nakayama subvariety of $D$ if and only if $D$ is big.
The following lemma guarantees the existence of smooth Nakayama subvarieties.
It also shows that a Nakayama subvariety is not unique in general.

\begin{proposition}[cf. {\cite[Lemma V.2.11]{nakayama}}]\label{prop-gen subvar=naka}
Let $D$ be a divisor on a variety $X$ with $\kappa(D) \geq 0$.
Any general subvariety $U \subseteq X$ of dimension $\kappa(D)$ is a Nakayama subvariety of $D$.
In particular, $U=A_1 \cap \cdots \cap A_{n-\kappa(D)}$ is a Nakayama subvariety of $D$ where $A_1, \cdots, A_{n-\kappa(D)}$ are general members of sufficiently ample linear systems.
\end{proposition}

\begin{proof}
Consider a dominant rational map $\varphi_m : X \dashrightarrow Z_m$ induced by a complete linear series $|\lfloor mD \rfloor |$ for any integer $m>0$ such that $|\lfloor mD \rfloor | \neq \emptyset$. We have $\dim Z_m \leq \kappa(D) = \dim U$. Let $f : Y \to X$ be the blow-up at $U$ with the exceptional divisor $E$. Since $U \subseteq X$ is general, $f^*(\lfloor mD \rfloor) - k E$ is not effective for any $k > 0$. Thus $H^0(X,\mc I_{U} \otimes \mathcal{O}_X(\lfloor mD \rfloor))=0$, and hence, $U$ is a Nakayama subvariety of $D$. The latter statement also follows from this argument.
\end{proof}

Nakayama subvarieties have the following properties.

\begin{lemma}\label{lem-same nakayama loci}
If $D, D'$ are divisors on a variety $X$ such that $\kappa(D), \kappa(D') \geq 0$ and $D \sim_\R D'$, then they have the same Nakayama subvarieties.
\end{lemma}

\begin{proof}
The given conditions imply that $\kappa(D)= \kappa(D')$.
Suppose that $U$ is a Nakayama subvariety of $D$, but it is not a Nakayama subvariety of $D'$.
Then for some integer $m>0$, there exist an effective $\Z$-divisor $E$ such that $E \sim \lfloor mD' \rfloor$ and an irreducible component $E_0$ of $E$ that contains $U$. This means that there exists an element $D'' \in |D'|_{\R} = |D|_{\R}$ such that $E_0$ is a component of $D''$. By Lemma \ref{r-divisor},
we can find an integer $m' > 0$ such that some element of $\lfloor m'D \rfloor$ contains $E_0$ as an irreducible component.
Then such an element should also contain $U$. Thus we have a contradiction to our assumption that $U$ is a Nakayama subvariety of $D$.
\end{proof}

Let $V$ be a smooth subvariety of a variety $X$.
A birational morphism $f:Y\to X$ is said to be \emph{$V$-birational} if $V$ is not contained in the image of the $f$-exceptional locus.
We denote by $\widetilde{V}$ the proper transform of $V$ on $Y$. Then the pair $(Y, \widetilde{V})$ is called a \emph{$V$-birational model} of $X$ (see \cite[Definition 2.10]{lehmann-nu}).

\begin{proposition}
Let $U \subseteq X$ be a Nakayama subvariety of a divisor $D$ on a variety $X$ with $\kappa(D) \geq 0$.
If $f:Y\to X$ is a $U$-birational morphism, then $\widetilde{U}$ is also a Nakayama subvariety of $f^*D$.
\end{proposition}

\begin{proof}
Suppose that $\widetilde{U}$ is not a Nakayama subvariety of $f^*D$. Then there exists an integer $m>0$ and an effective $\Z$-divisor $E$ on $Y$ such that $E \sim \lfloor mf^*D \rfloor$ and $\widetilde{U} \subseteq \Supp(E)$. Then $f_*E \sim f_*\lfloor mf^*D \rfloor =\lfloor mD \rfloor$ and $U \subseteq \Supp(f_*E)$. By definition, this is a contradiction.
\end{proof}

\begin{theorem}\label{lem-D big on naka}
Let $D$ be a divisor on a variety $X$ with $\kappa(D) \geq 0$, and $U\subseteq X$ be its Nakayama subvariety.
Then $D|_U$ is big.
\end{theorem}

\begin{proof}
By the definition of Nakayama subvariety, we have
$$h^0(U,\mathcal{O}_U (\lfloor mD|_U \rfloor)) \geq h^0(X,\mathcal{O}_X(\lfloor mD \rfloor))$$
for every integer $m \geq 0$. Since $h^0(X,\mc O_X(\lfloor mD \rfloor)) \sim m^{\kappa(D)}$ for $m\gg0$, the assertion follows.
\end{proof}

\subsection{Positive volume subvariety}\label{subsec-posvolloc}\hfill

In this subsection, we introduce and study positive volume subvarieties of divisors, which are closely related to the numerical Iitaka dimension and the restricted volume.
First, we review the numerical Iitaka dimension.

\begin{definition}\label{def-kappanu}
Let $D$ be a divisor on a variety $X$.
We define the \emph{numerical Iitaka dimension} of $D$ as the nonnegative integer
$$
\kappanu(D):= \max\left\{k \in \Z_{\geq0} \left|\; \limsup_{m \to \infty} \frac{h^0(X,\mathcal{O}_X(\lfloor mD \rfloor + A))}{m^k}>0 \right.\right\}
$$
for some fixed ample Cartier divisor $A$ if $h^0(X,\mathcal{O}_X(\lfloor mD \rfloor + A))\neq\emptyset$ for infinitely many $m>0$ and we let $\kappanu(D):=-\infty$ otherwise.
\end{definition}

The numerical Iitaka dimension $\kappanu(D)$ depends only on the numerical class $[D]\in\N^1(X)_{\R}$. Note that $D$ is pseudoeffective if and only if $\kappanu(D) \geq 0$.
One can easily check that $\kappa(D) \leq \kappanu(D)$ holds and the inequality is strict in general (see \cite[Example 6.1]{lehmann-nu}).
However, if $\kappanu(D)=\dim X$, then $\kappa(D)=\dim X$.
See  \cite{nakayama}, \cite{lehmann-nu}, and \cite{Eck} for detailed properties of $\kappa$ and $\kappanu$.

The following theorem relates the numerical Iitaka dimension $\kappanu(D)$ and $\vol^+_{X|V}(D)$.

\begin{theorem}[{\cite[Theorem 1.1]{lehmann-nu}}]\label{thrm-vol+=kappanu}
Let $D$ be a pseudoeffective divisor on a variety $X$.
Then we have
$$
\kappanu(D)=\max\{\dim W \mid \vol_{X|W}^+(D)>0\;\text{ and }\; W\not\subseteq\bm(D)\}.
$$
\end{theorem}

The numerical Iitaka dimension $\kappanu(D)$ actually coincides with many other invariants defined with $D$.
For details, see \cite{lehmann-nu} and \cite{Eck}.

Now we define the positive volume subvariety associated to a pseudoeffective divisor.

\begin{definition}\label{def-pos vol locus}
Let $D$ be a pseudoeffective divisor on a variety $X$.
A subvariety $V\subseteq X$ of dimension $\kappanu(D)$
such that $\vol_{X|V}^+(D)>0$ and $V\not\subseteq\bm(D)$ is called a \emph{positive volume subvariety} of $D$.
(In other words, a positive volume subvariety $V \subseteq X$ is the subvariety at which the maximum in Theorem \ref{thrm-vol+=kappanu} is attained.)
\end{definition}

\begin{remark}
In \cite[Lemma 4.3]{CPW}, we show that $\vol_{X|V}^+(D)>0$ implies $V \not\subseteq \bm(D)$. Thus  the condition $V\not\subseteq\bm(D)$ in the definition of the positive volume subvariety is redundant.
\end{remark}

\begin{example}
Let $S$ be a relatively minimal rational elliptic surface with a reducible singular fiber, and $E$ a $(-2)$-curve in a singular fiber. Then $E \not\subseteq \bm(-K_S)$ but $\vol_{S|E}^+(-K_S)=0$.
In this example, we see that $V \not\subseteq \bm(D)$ and $\dim V = \kappanu(D)$ does not imply $\vol_{X|V}^+(D) >0$.
\end{example}

By definition, if $D$ is a big divisor on $X$, then $V=X$ is the unique positive volume subvariety of $D$. Theorem \ref{thrm-vol+=kappanu} guarantees the existence of a positive volume subvariety in general. Furthermore, we can also find a smooth positive volume subvariety by the following. It also shows that a positive volume subvariety is not unique in general.

\begin{proposition}\label{general pos vol loc}
Let $D$ be a pseudoeffective divisor on a variety $X$ of dimension $n$.
If $\kappanu(D)<n$, then 
the subvariety $V:=A_1 \cap \cdots \cap A_{n - \kappanu(D)}$ satisfies $\vol_{X|V}^+(D)>0$, where  $A_1, \cdots, A_{n-\kappanu(D)}$ are general members of sufficiently ample linear systems. In particular, $V$ is a positive volume subvariety of $D$.
\end{proposition}

\begin{proof}
By Theorem \ref{thrm-vol+=kappanu}, there exists a positive volume subvariety $W$ of $D$. In particular, we have $\vol_{X|W}^+(D)>0$.
We can take a sequence $\{ H_i \}_{i \in \Z_{\geq 0}}$ of ample divisors on $X$ such that each $D+H_i$ is a $\Q$-divisor and $H_i \to 0$ as $i \to \infty$.
For a large and sufficiently divisible integer $k$, choose $\kappanu(D)$ general divisors $E_{k,i}^{1}, \cdots, E_{k,i}^{\kappanu(D)} \in |k(D+ H_i)|$.
Since $A_1, \cdots, A_{n-\kappanu(D)}$ are sufficiently ample divisors, we have
$$
\# (V \cap E_{k,i}^1 \cap \cdots \cap E_{k,i}^{\kappanu(D)} \setminus \text{SB}(D+H_i)) \geq \# (W \cap E_{k,i}^1 \cap \cdots \cap E_{k,i}^{\kappanu(D)} \setminus \text{SB}(D+H_i)).
$$
Thus by applying \cite[Theorem B]{elmnp-restricted vol and base loci}, we obtain
$$
\begin{array}{rl}
 \vol_{X|V}(D+ H_i) &=\lim\limits_{k \to \infty} \frac{\# (V \cap E_{k,i}^1 \cap \cdots \cap E_{k,i}^{\kappanu(D)} \setminus \text{SB}(D+H_i))}{k^{\kappanu(D)}}\\
&\geq  \lim\limits_{k \to \infty} \frac{\# (W \cap E_{k,i}^1 \cap \cdots \cap E_{k,i}^{\kappanu(D)} \setminus \text{SB}(D+H_i))}{k^{\kappanu(D)}}=\vol_{X|W}(D+ H_i).
\end{array}
$$
This implies that $\vol_{X|V}^+(D) \geq \vol_{X|W}^+(D)>0$.
Since $A_1, \cdots, A_{n-\kappanu(D)}$ are sufficiently ample, we also have $V \not\subseteq \bm(D)$.
\end{proof}

We prove some notable properties of positive volume subvariety.

\begin{lemma}\label{lem-same +vol locus}
If $D$ and $D'$ are pseudoeffective divisors on a variety $X$ such that $D\equiv D'$, then they have the same positive volume subvarieties.
\end{lemma}
\begin{proof}
Note first that $\bm (D) = \bm (D')$.
The statement then follows from the fact that $\vol^+_{X|V}(D)$ depends only on the numerical class $[D]$  for any $V \not\subseteq \bm(D)=\bm(D')$.
\end{proof}

\begin{proposition}\label{v-bir posvol loc}
Let $V\subseteq X$ be a positive volume subvariety of a pseudoeffective divisor $D$ on a variety $X$.
If $f:Y\to X$ is a $V$-birational morphism, then $\widetilde{V}$ is also a positive volume subvariety of $f^*D$.
\end{proposition}

\begin{proof}
The birational transform $\widetilde V\subseteq Y$ of $V$ has dimension $\dim \widetilde{V}=\dim V=\kappanu(D)=\kappanu(f^*D)$.
If $\widetilde V\subseteq\bm(f^*(D))$, then $V=f(\widetilde V)\subseteq f(\bm(f^*(D))=\bm(D)$.
Thus $\widetilde V\not\subseteq\bm(f^*(D))$.
It is enough to show that $\vol_{Y|\widetilde V}^+(f^*D) >0$.
Since $V\not\subseteq\bm(D)$ by definition,
for any ample divisor $A$ on $X$, we have $V\not\subseteq\bp(D+A)$.
We can also check that $\widetilde V\not\subseteq\bp(f^*D+f^*A)$ using \cite[Proposition 2.3]{BBP}.
Note that there is an effective divisor $E$ on $Y$ such that $f^*A-E$ is ample. Thus we have
$$
\vol_{Y|\widetilde V}^+(f^*D) = \lim_{\epsilon \to 0+} \vol_{Y|\widetilde V} (f^*D + \epsilon(f^*A-E)).
$$
The difference $| \vol_{Y|\widetilde V} (f^*D + \epsilon(f^*A-E)) - \vol_{Y|\widetilde V}(f^*D+\epsilon f^*A)|$ can be arbitrarily small when $\epsilon>0$ is sufficiently small.
On the other hand, by \cite[Lemma 2.4]{elmnp-restricted vol and base loci}, we have
$$
\vol_{X|V}^+(D)=\lim_{\epsilon \to 0+}\vol_{X|V}(D+\epsilon A)=\lim_{\epsilon \to 0+}\vol_{Y|\widetilde V}(f^*D+\epsilon f^*A).
$$
Since $\vol_{X|V}^+(D) >0$, it follows that $\vol_{Y|\widetilde V}^+(f^*D) >0$.
\end{proof}

It is known that $D|_V$ is pseudoeffective for any subvariety $V\subseteq X$ such that $V\not\subseteq\bm(D)$. Thus $\kappa(D|_V)\geq 0$. Furthermore, we have the following.

\begin{theorem}\label{D on +vol big}
Let $V$ be a positive volume subvariety of a pseudoeffective divisor $D$ on a variety $X$.
Then $\vol_{V} (D|_V)>0$.
In particular, $D|_V$ is big on $V$ and $\kappanu(D)=\kappanu(D|_V)=\dim V$.
\end{theorem}

\begin{proof}
By taking a suitable $V$-birational model and considering Proposition \ref{v-bir posvol loc}, we may assume that $V$ is smooth.
We can take a sequence of ample divisors $A_i$ such that $D+A_i$ are $\Q$-divisors and $A_i\to 0$ as $i\to \infty $.
Since $V\subseteq X$ is a positive volume subvariety of $D$, there exists a constant $C_0$ such that for any large integer $k>0$, we have
$$
\limsup_{m\to\infty}\frac{h^0(X|V,m(D+A_k))}{m^{\kappanu(D)}}>C_0
$$
where $m$ is taken over all positive integers such that $m(D+A_k)$ is a $\Z$-divisor.
Note that $h^0(X|V,m(D+A_k))\leq h^0(V,m(D|_V+A_k|_V))$.
Thus for any large integer $k>0$, we also have
$$
\limsup_{m\to\infty}\frac{h^0(V,m(D|_V+A_k|_V))}{m^{\kappanu(D)}}> C_0.
$$
By the continuity of the volume function $\vol_V$, we obtain $\vol_V (D|_V)>0$. Thus $D|_V$ is big on $V$.
\end{proof}

\section{Okounkov body of a pseudoeffective divisor}\label{oksec}

In this section, we introduce and study the valuative Okounkov body $\okval_{Y_\bullet}(D)$ and the limiting Okounkov body $\oklim_{Y_\bullet}(D)$ associated to a divisor $D$ on a variety $X$.

\subsection{Okounkov body $\okbd_{Y_\bullet}(D)$}\label{ok of LM}\hfill

Let $X$ be a  (possibly singular) projective variety of dimension $n$, and $D$ be an $\R$-divisor on $X$ such that the $\R$-linear system $|D|_{\R}$ is nonempty.
In Introduction, we have defined the function $\nu_{Y_\bullet}: |D|_{\R}\to \R^n$ for a given admissible flag $Y_\bullet$ on $X$.
We also recall that the \emph{Okounkov body} $\okbd_{Y_\bullet}(D)$ of a divisor $D$ with respect to an admissible flag $Y_\bullet$ is defined as
$$
\okbd_{Y_\bullet}(D) := \text{the closure of the convex hull of the image $\nu_{Y_\bullet}(|D|_\R)$ in $\R^n_{\geq 0}$}.
$$
We remark that such construction of $\okbd_{Y_\bullet}(D)$ is equivalent to the one given in \cite{lm-nobody} where $\okbd_{Y_\bullet}(D)$ is defined as the natural extension of the Okounkov body defined for $\Z$-divisors using the homogeneity and the continuity of the Okounkov bodies.
See \cite{lm-nobody} for details.

In this Subsection, we also recall  from \cite{lm-nobody} the construction of the Okounkov body of  a graded linear series $W_{\bullet}$ associated to $D$.

Let $D$ be a Cartier $\Z$-divisor on $X$ such that $H^0(X,\mc O_X(D))\neq 0$, and fix an admissible flag $Y_\bullet$ on $X$.
For a non-zero section $s \in H^0(X, \mathcal{O}_X(D)) \setminus \{ 0 \}$, we can find an effective divisor $D'=\text{div}(s) \sim D$.
Thus we can also define the following function
$$
\nu_{Y_\bullet}:  H^0(X, \mathcal{O}_X(D)) \setminus \{ 0 \} \to \Z^n_{\geq 0},\;\;\nu_{Y_\bullet}(s):=\nu_{Y_\bullet}(D').
$$
Recall that a \emph{graded linear series} $W_{\bullet}(D)=\{ W_m:=W_m(D) \}_{m \geq 0}$ associated to $D$ consists of subspaces $W_m \subseteq H^0(X,\mathcal{O}_X(mD))$ with $W_0=\C$
 satisfying $W_m \cdot W_l \subseteq W_{m+l}$ for all $m,l \geq 0$.
 The \emph{volume} of $W_\bullet$ is defined as
$$
\vol_X(W_{\bullet}) := \lim_{m \to \infty} \frac{\dim W_m}{m^n/n!}
$$
Note that the above limit exists by \cite[Theorem 2.13 and Remark 2.14]{lm-nobody}.
We define the \emph{graded semigroup} of $W_{\bullet}$ as
$$
\Gamma(W_{\bullet}):=\{(\nu(s),m) \mid  \; W_m\neq \{0\} \text{ and } s\in W_m\setminus\{0\} \;\} \subseteq \mathbb Z_{\geq0}^n\times \mathbb Z_{\geq0}.
$$

\begin{definition}\label{def-conbd}
Let $W_{\bullet}$ be a graded linear series associated to a Cartier $\Z$-divisor $D$ on a variety $X$ of dimension $n$.
We can associate the convex body $\Delta_{Y_\bullet}(W_{\bullet})$
called the \emph{Okounkov body} of a graded linear series $W_{\bullet}$ with respect to an admissible flag $Y_\bullet$ on $X$ as follows:
$$
\Delta_{Y_\bullet}(W_\bullet):= \Sigma(\Gamma(W_{\bullet})) \cap (\R_{\geq0}^n \times \{ 1 \} ) \subseteq \R^n_{\geq 0} \times \{ 1 \} = \R^n_{\geq 0}
$$
where $\Sigma(\Gamma(W_{\bullet}))$ denotes the closure of the convex cone in $\R_{\geq0}^n\times \R_{\geq0}$ spanned by $\Gamma(W_{\bullet})$.
\end{definition}

Note also that $\Delta_{Y_\bullet}(W_\bullet)$ is a convex subset in $\R^n$, that is, $\Delta_{Y_\bullet}(W_\bullet)\subseteq\R^n_{\geq 0}$.
If $W_\bullet$ is a complete graded linear series (that is, $W_m=H^0(X,\mc O(mD))$ for each $m$) of a Cartier $\Z$-divisor $D$, then we write $W_\bullet=W_\bullet(D)$ and we have
$$
\okbd_{Y_\bullet}(D)=\okbd_{Y_\bullet}(W_\bullet(D)).
$$

By construction, the Okounkov body $\okbd_{Y_\bullet}(D)$ of a divisor $D$ or the Okounkov body $\okbd_{Y_\bullet}(W_\bullet)$ of a graded linear series $W_\bullet$ depends on the choice of the admissible flag $Y_\bullet$.

\begin{remark}\label{remk-partial flag}
Let $Y_\bullet$ be an admissible flag on a variety $X$ of dimension $n$.
For each integer $k$ such that $0 \leq k  \leq n$, we can define the $k$-th \emph{partial flag} $Y_{k\bullet}$ of $Y_\bullet$ as
$$
Y_{k\bullet}: Y_{k} \supseteq Y_{k+1}\supseteq \cdots\supseteq Y_n=\{x\}.
$$
Then $Y_{k\bullet}$ is an admissible flag on a projective variety $Y_{k}$ of dimension $n-k$.
If $D$ is an effective divisor on $X$ such that $\ord_{Y_k}(D)=0$, then we obtain an effective divisor $D|_{Y_k}$ on $Y_k$. As above, we can define
$$
\nu_{Y_{k\bullet}}(D) := (\nu_{k+1}(D), \cdots, \nu_n(D)) \in \R^{n-k}_{\geq 0}.
$$
If $\ord_{Y_k}(D) > 0$, then we set $\nu_{Y_{k\bullet}}(D)=(0, \cdots, 0) \in \R^{n-k}_{\geq 0}$.
The \emph{$k$-th partial Okounkov body} of a divisor $D$ on $X$ with respect to the $k$-th partial flag $Y_{k\bullet}$ is defined as
$$
\okbd_{Y_{k\bullet}}(D):= \text{the closure of the convex hull of $\nu_{Y_{k\bullet}}(|D|_{\R})$ in $\R^{n-k}_{\geq 0}$.}
$$
These partial flags and partial Okounkov bodies will be useful in the next subsections.
\end{remark}

Under the following conditions, the Okounkov body $\Delta_{Y_\bullet}(W_\bullet)$ behaves well.

\begin{definition}[{\cite[Definitions 2.5 and 2.9]{lm-nobody}}]\label{condition CD}\hfill

\begin{enumerate}[leftmargin=0.3cm,itemindent=0.5cm]
\item We say that a graded linear series $W_{\bullet}$ satisfies \emph{Condition} (B) if $W_m \neq 0$ for all $m \gg 0$, and for all sufficiently large $m$, the rational map $\varphi_m : X \dashrightarrow \P(W_m)$ defined by $|W_m|$ is birational onto its image.
\item We say that a graded linear series $W_{\bullet}$ satisfies \emph{Condition} (C) if
\begin{enumerate}
 \item[i.] for every $m \gg 0$, there exists an effective divisor $F_m$ such that the divisor $A_m := mD-F_m$ is ample, and
 \item[ii.] for all sufficiently large $p$, we have
$$
H^0(X, \mc{O}_X(pA_m)) \subseteq W_{pm} \subseteq H^0(X,\mc{O}_X(pmD)).
$$
\end{enumerate}
\end{enumerate}
\end{definition}

If $W_{\bullet}$ is complete, that is, $W_m = H^0(X, \mathcal{O}_X(mD))$ for all $m \geq 0$ and $D$ is big, then it automatically satisfies Condition (C).

\begin{theorem}[{\cite[Theorem 2.13]{lm-nobody}}]\label{thrm-okbd of big}
Suppose that a graded linear series $W_\bullet$ satisfies Condition (B) or Condition (C).
Then for any admissible flag $Y_\bullet$ (with a general choice of the point $Y_n=\{x\}$ when $W_\bullet$ only satisfies Condition (B)), we have
$$
\dim \Delta_{Y_\bullet}(W_{\bullet})=\dim X (=n) \text{\;\; and \;\;}  \vol_{\R^n}(\Delta_{Y_\bullet}(W_{\bullet}))=\frac{1}{n!}\cdot\vol_X (W_{\bullet}).
$$
\end{theorem}

Note that \cite[Theorem 2.13]{lm-nobody} also requires Condition (A) (\cite[Definition 2.4]{lm-nobody}),
but it is automatically satisfied in our situation since we always assume that our variety $X$ is projective. The smoothness assumption on $X$ is not necessary for this theorem.

\begin{remark}\label{ex-okbd not num}
It is well known by \cite[Proposition 4.1]{lm-nobody} that for a fixed admissible flag $Y_\bullet$ on $X$,
if $D$ is a big divisor on $X$, then $\okbd_{Y_\bullet}(D)$ depends only on the numerical class of $D$.
If $D$ is not big, then it is no longer true.
See Remark \ref{ex-okval-not num}.
\end{remark}


We will use the following result, which generalizes \cite[Theorem 4.26]{lm-nobody}.
Many of our subsequent results will rely on this nontrivial theorem.
The detailed proof is given in \cite{CPW}.

\begin{theorem}[{\cite[Theorem 1.1]{CPW}}]\label{thrm-okbd slice}
Let $D$ be a big divisor on a smooth projective variety $X$ of dimension $n$.
Suppose that $Y_\bullet$ is an admissible flag on $X$ such that $Y_{k} \not\subseteq \bp(D)$. Then we have
$$\okbd_{Y_{\bullet}}(D) \cap (\{0 \}^k \times \R^{n-k}_{\geq 0})=\okbd_{Y_{k\bullet}}(D).$$
\end{theorem}

\begin{remark}
If $Y_k \subseteq \bp(D)$, then the conclusion of Theorem \ref{thrm-okbd slice} does not hold in general. For example, following \cite[Example 5.10]{elmnp-restricted vol and base loci}, we consider the blow-up $\pi : X \to \P^3$ of $\P^3$ at a line with exceptional divisor $E$. Let $D \in |\pi^*\mathcal{O}_{\P^3}(1)|$ be an effective divisor, and $C$ be a smooth curve of type $(2,1)$ in $E \simeq \P^1 \times \P^1$.
Note that $D$ is nef and big and $C \subseteq \bp(D)=E$.
Fix an admissible flag
$$
Y_\bullet : X=Y_0 \supseteq E=Y_1 \supseteq C=Y_2 \supseteq Y_3=\{ x \}
$$
on $X$ where $x$ is any point on $C$. We may regard both $\okbd_{Y_{\bullet}}(D) \cap (\{0 \}^2 \times \R^{1}_{\geq 0})$ and $\okbd_{Y_{2\bullet}}(D)$ as convex subsets in $\R^1_{\geq 0}$. We then have
$$
\okbd_{Y_{\bullet}}(D) \cap (\{0 \}^2 \times \R^{1}_{\geq 0}) = \{ x \in \R^1_{\geq 0} \mid 0 \leq x \leq 2 \} \supsetneq
 \{ x \in \R^1_{\geq 0} \mid 0 \leq x \leq 1 \} = \okbd_{Y_{2\bullet}}(D).
$$
\end{remark}

\subsection{Valuative Okounkov body $\okval_{Y_\bullet}(D)$}\label{okbdnaksubsec}\hfill

Throughout the subsection, $X$ is a smooth projective variety of dimension $n$.

\begin{definition}\label{okval}
Let $D$ be a divisor on a variety $X$ of dimension $n$ such that $|D|_{\R} \neq \emptyset$.
The \emph{valuative Okounkov body} $\okval_{Y_\bullet}(D)$ associated to $D$ with respect to an admissible flag $Y_\bullet$ on $X$ is defined as
$$\okval_{Y_\bullet}(D):=\Delta_{Y_\bullet}(D)\subseteq\R^n_{\geq 0}.$$
For a divisor $D$ with $|D|_{\R} = \emptyset$, we define $\okval_{Y_\bullet}(D):=\emptyset$.
\end{definition}

\begin{remark}
Note that the constructions for the valuative Okounkov body $\okval_{Y_\bullet}(D)$ and for the usual Okounkov body $\okbd_{Y_\bullet}(D)$ are the same. However, for a non-big divisor $D$, we will use the notation $\okval_{Y_\bullet}(D)$ in order to emphasize special properties of valuative Okounkov body  $\okval_{Y_\bullet}(D)$.
\end{remark}

Clearly,  the valuative Okounkov body $\okval_{Y_\bullet}(D)$ of a divisor $D$ depends on the choice of the admissible flag $Y_\bullet$.
However, we will see below that if $\kappa(D) \geq 0$ and $Y_\bullet$ contains a Nakayama subvariety $U=Y_{n-\kappa(D)}$ of $D$, then the valuative Okounkov body $\okval_{Y_\bullet}(D)$ depends only on the choice of a Nakayama subvariety of $D$ and the $(n-\kappa(D))$-th partial flag $Y_{n-\kappa(D)\bullet}$ on $U$.

Let $D$ be a divisor on $X$ with $\kappa(D) \geq 0$, and fix an admissible flag $Y_\bullet$ on $X$ containing a Nakayama subvariety $U=Y_{n-\kappa(D)}$ of $D$. As in Remark \ref{remk-partial flag}, we can also consider $\nu_{Y_{n-\kappa(D)\bullet}}(D') \in \R^{\kappa(D)}_{\geq 0}$ for any $D' \in |D|_{\R}$.

\begin{proposition}\label{okval=0 in nu}
Under the notations as above, for any $D' \in |D|_{\R}$, we have
$$
\nu_{Y_\bullet}(D')=(\underbrace{0, \ldots, 0}_{n-\kappa(D)}, \nu_{Y_{n-\kappa(D)\bullet}}(D')).
$$
In particular, $\okval_{Y_\bullet}(D) \subseteq \{0\}^{n-\kappa(D)} \times \R^{\kappa(D)}_{\geq 0}$, and so we can regard $\okval_{Y_\bullet}(D)$ as a subset of $\R^{\kappa(D)}_{\geq 0}$.
\end{proposition}

\begin{proof}
By Lemma \ref{lem-same nakayama loci}, $U=Y_{n-\kappa(D)}$ is also a Nakayama subvariety of $D'$.
By the definition of Nakayama subvariety, we have
$$H^0(X,\mathcal{I}_{U} \otimes \mathcal{O}_X(\lfloor mD' \rfloor))=0$$
for all $m \geq 0$. This implies that $U \not\subseteq \text{Supp}(D')$.
Since $Y_k\supseteq U$ for $1 \leq k \leq n-\kappa(D)$, it follows that $Y_k \not\subseteq \text{Supp}(D')$. Thus we obtain $\nu_k(D')=0$ for $1 \leq k \leq n-\kappa(D)$, and the assertions immediately follow.
\end{proof}

Now let $D$ be a $\Z$-divisor on $X$ such that $\kappa(D)\geq 0$, and fix a Nakayama subvariety $U\subseteq X$ of $D$. By Theorem \ref{lem-D big on naka}, $D|_U$ is a big divisor on $U$.
Consider the \emph{restricted complete graded linear series} $W_\bullet(D|U)$ of a $\Z$-divisor $D$ along $U$ which is given by
$$
W_m(D|U):=H^0(X|U,mD)
$$
for each $m \geq 0$.
Recall that $H^0(X|U,mD)$ is the image of the natural restriction map $\varphi_U:H^0(X,\mathcal{O}_X(mD)) \to H^0(U,\mathcal{O}_U(mD|_U))$.
Note that $W_\bullet(D|U)$ is a graded linear subseries of $W_\bullet(D|_U)$ on $U$.


\begin{remark}\label{remk-another okval}
Let $Y_\bullet$ be an admissible flag containing a Nakayama subvariety $U$ of a $\Z$-divisor $D$ on $X$ such that $\kappa(D)\geq 0$.
Then the valuative Okounkov body $\okval_{Y_\bullet}(D)$ can be constructed alternatively as follows.
Consider the  $(n-\kappa(D)$)-th partial flag  $Y_{n-\kappa(D)\bullet}$ on $U$.
Since $W_\bullet(D|U)$ is a graded linear subseries associated to $D|_U$ on $U$,
we can define the Okounkov body
$$
\Delta_{Y_{n-\kappa(D)\bullet}}(W_\bullet(D|U)) \subseteq \R^{\kappa(D)}_{\geq 0}
$$
as in Definition \ref{def-conbd}.
By regarding  $\okval_{Y_\bullet}(D)$ as a subset of $\R^{\kappa(D)}$ (Proposition \ref{okval=0 in nu}), we have
$$
\okval_{Y_\bullet}(D) =\Delta_{Y_{n-\kappa(D)\bullet}}(W_\bullet(D|U)).
$$
Thus $\Delta_{Y_{n-\kappa(D)\bullet}}(W_\bullet(D|U))$ gives an alternative construction of $\okval_{Y_\bullet}(D)$.
\end{remark}

The following is the main property of $\okval_{Y_\bullet}(D)$.

\begin{theorem}\label{okval-main}
Let $D$ be a divisor such that $\kappa(D)\geq 0$ on a smooth projective variety $X$, and fix an admissible flag $Y_\bullet$ containing a Nakayama subvariety $U$ of $D$ such that $Y_n=\{x\}$ is a general point. Then $\okval_{Y_\bullet}(D) \subseteq \{ 0 \}^{n-\kappa(D)} \times \R^{\kappa(D)}$. Furthermore, we have
$$
\dim \okval_{Y_\bullet}(D)=\kappa(D) \text{\;\; and \;\;} \vol_{\R^{\kappa(D)}}(\okval_{Y_\bullet}(D))=\frac{1}{\kappa(D)!}\vol_{X|U}(D)
$$
where $\okval_{Y_\bullet}(D)$ is regarded as a convex subset of $\R^{\kappa(D)}$.
\end{theorem}

\begin{proof}
The first part of the assertions is shown in Proposition \ref{okval=0 in nu}. We now prove the second part. We first consider the case where $D$ is a $\Z$-divisor.
We can easily check that $\okval_{Y_\bullet}(mD)=\frac{1}{m}\okval_{Y_\bullet}(D)$ for any integer $m>0$. Thus we may assume that the constant $m_0$ in Remark \ref{remk-kappa=m^k} is equal to $1$. In particular, $h^0(X, \mathcal{O}_X(mD)) >0$ for any integer $m>0$.
By Remark \ref{remk-another okval}, we have
$\okval_{Y_\bullet}(D) =\Delta_{Y_{n-\kappa(D)\bullet}}(W_\bullet(D|U))$.
By the properties of a Nakayama subvariety, there exits an integer $m_0$ such that $h^0(X,\mathcal{O}_X(mm_0D))=\dim W_{mm_0} \sim m^{\kappa(D)}$ for all $m \gg 0$
(cf. Remark \ref{remk-kappa=m^k}) and $\dim U=\kappa(D)$.
Furthermore, $Y_n=\{x\}$ is assumed to be general. By \cite[Remark 2.8]{lm-nobody}, the assertions follow from Theorem \ref{thrm-okbd of big}.

Now, we assume that $D$ is a $\Q$-divisor. There exists a sufficiently divisible integer $m>0$ such that $mD$ is a $\Z$-divisor. We can also see that $\okval_{Y_\bullet}(D)=\frac{1}{m}\okval_{Y_\bullet}(mD)$. Thus the assertions for $\Q$-divisors follow from the $\Z$-divisor case.

Finally, we treat the case where $D$ is an $\R$-divisor. Consider a sequence of $\Q$-divisors $D_m:=\frac{\lfloor mD \rfloor}{m}$ converging to $D$ as $m \to \infty$. Then $E_m:=D-D_m=\frac{mD - \lfloor mD \rfloor}{m}$ forms a sequence of effective $\R$-divisors converging to $0$ as $m \to \infty$.
We have
$$
\okval_{Y_\bullet}(D_m) + \okval_{Y_\bullet}(E_m) \subseteq \okval_{Y_\bullet}(D) ~~\text{ and }~~ \lim_{m \to \infty} \okval_{Y_\bullet}(E_m) = \{ \text{the origin of $\R^{\kappa(D)}$} \}
$$
as convex subsets in $\R^{\kappa(D)}_{\geq 0}$.
Take any $D' \in |D|_{\R}$. For any sufficiently small number $\eps > 0$, by applying Lemma \ref{r-divisor}, we can find $D_m' \in |D_m|_{\Q} := \{D_m' \mid D_m \sim_{\Q} D_m' \geq 0 \}$ for a sufficiently large and divisible integer $m>0$  such that
$$
|\nu_i - \nu_i'| < \eps~~ \text{ for all $i$ }
$$
where $\nu_{Y_\bullet}(D')=(\nu_1, \cdots, \nu_n)$ and $\nu_{Y_\bullet}(D_m')=(\nu_1', \cdots, \nu_n')$.
Observe that if $m \mid m'$, then $D_m \leq D_{m'}$.
Since $m>0$ is sufficiently large, $D_{m'}-D_m = E_m - E_{m'}$ is arbitrarily small.
Thus for a sufficiently large and divisible integer $m'>0$ with $m \mid m'$, we can also find $D_{m'}'  \in |D_{m'}|_{\Q}$ such that
$$
|\nu_i - \nu_i''| < \eps  ~~\text{ for all $i$ }
$$
where  $\nu_{Y_\bullet}(D_{m'}')=(\nu_1'', \cdots, \nu_n'')$.
Therefore, by considering  sufficiently divisible integers $m>0$, we may assume that
$$
\lim_{m \to \infty} \okval_{Y_\bullet}(D_{m}) = \okval_{Y_\bullet}(D).
$$
On the other hand, by definition of the Iitaka dimension, we have $\kappa(D_m)=\kappa(D)$ for a sufficiently large $m>0$.
By definition of the restricted volume, we have
$$
\vol_{X|U}(D)=\lim_{m \to \infty} \vol_{X|U}(D_{m}).
$$
Note that $U$ is a Nakayama subvariety of $D_m$ for a sufficiently large $m>0$.
Thus the assertions for $\R$-divisors follow from the $\Q$-divisor case.
\end{proof}

\begin{remark}\label{ex-okval-not num}
It is natural to ask whether $\okval_{Y_\bullet}(D)$ is a numerical invariant of $D$ or not.
If $D$ is a big divisor on $X$, then by Remark \ref{ex-okbd not num}, $\okval_{Y_\bullet}(D)=\okbd_{Y_\bullet}(D)$ depends only on the numerical class of $D$.
However, in general, we may have $\kappa(D)<\kappa(D')$ even if  $D \equiv D'$ (see \cite[Example 6.1]{lehmann-nu}).
Thus by choosing an admissible flag $Y_\bullet$ containing Nakayama subvarieties of $D$ and $D'$, we obtain $\okval_{Y_\bullet}(D)\neq\okval_{Y_\bullet}(D')$.
\end{remark}

\begin{remark}\label{effuptoR}
If $D$ is only effective up to $\R$-linear equivalence, then Theorem \ref{okval-main} does not hold in general. As in Remark \ref{remk-kappa not num}, take two distinct points $P,Q$ on $\P^1$ and an irrational number $a$. Set $D:=a(P-Q)$. Then $D \sim_{\R} 0$ and $\okval_{Y_\bullet}(D)=\{ 0 \in \R^1 \}$. However, we have $\kappa(D)=-\infty$.
\end{remark}

On the other hand, we have the following.

\begin{proposition}\label{okvallin}
If two divisors $D$, $D'$ satisfy $D \sim_\R D'$, then $\okval_{Y_\bullet}(D)=\okval_{Y_\bullet}(D')$ with respect to any admissible flag $Y_\bullet$.
\end{proposition}
\begin{proof}
It is trivial by definition.
\end{proof}

\begin{remark}\label{rem-okvallin}
The converse of Proposition \ref{okvallin} is false even if $D$ is big. Indeed, for two big divisors $D, D'$ such that $D \equiv D'$ but $D \not\sim_{\R} D'$, we have $\okbd_{Y_\bullet}(D)=\okbd_{Y_\bullet}(D')$ for any admissible flag $Y_\bullet$ (see \cite[Proposition 4.1]{lm-nobody}). Moreover, we can also construct such non-big effective divisors.
For example, consider a minimal ruled surface $f : S \to C$ over a smooth projective curve of genus $g \geq 1$. Let $F_1:=f^*P$ and $F_2:=f^*Q$ where $P$ and $Q$ are two distinct points on $C$. Note that $F_1 \not\sim_{\R} F_2$. However, we claim that $\okval_{Y_\bullet}(F_1)=\okval_{Y_\bullet}(F_2)$ for any admissible flag $Y_\bullet$ on $S$. To see this, observe that a curve $Y_1$ is either a fiber of $f$ or dominating $C$ via $f$. In the first case, we have
$$
\okval_{Y_\bullet}(F_1)=\okval_{Y_\bullet}(F_2)=\{(x_1, x_2) \in \R^2 \mid 0 \leq x_1 \leq 1, x_2=0 \}.
$$
In the second case, note that $Y_1$ is a Nakayama subvariety of both $F_1$ and $F_2$. Furthermore, we have
$$
\vol_{S|Y_1}(F_1)=Y_1 \cdot F_1 = Y_1 \cdot F_2=\vol_{S|Y_1}(F_2),
$$
and hence, the claim follows.
\end{remark}

\subsection{Limiting Okounkov body $\oklim_{Y_\bullet}(D)$}\label{oklimsubsec}\hfill

Throughout the subsection, $X$ is a smooth projective variety of dimension $n$.

\begin{definition}\label{def-oklim on Y}
Let $D$ be a pseudoeffective divisor $D$ on a variety $X$ of dimension $n$.
The \emph{limiting Okounkov body} $\oklim_{Y_\bullet}(D)$ of $D$ with respect to an admissible flag $Y_\bullet$ is defined as
$$
\oklim_{Y_\bullet}(D) := \lim_{\eps \to 0+} \okbd_{Y_\bullet}(D+ \eps A)  = \bigcap_{\eps > 0} \okbd_{Y_\bullet}(D + \eps A) \subseteq \R^n_{\geq 0}
$$
where $A$ is an ample divisor. If $D$ is not a pseudoeffective divisor, then we define $\oklim_{Y_\bullet}(D):=\emptyset$.
\end{definition}

\begin{remark}
It is easy to see that the definition of the limiting Okounkov body does not depend on the choice of an ample divisor $A$. If $D$ is big, then $\oklim_{Y_\bullet}(D)$ coincides with $\okbd_{Y_\bullet}(D)$ for any admissible flag $Y_\bullet$ on $X$.
\end{remark}

We are mainly interested in the limiting Okounkov bodies $\oklim_{Y_\bullet}(D)$ when $Y_\bullet$ contains
contains a positive volume subvariety $V$ of $D$.
To investigate such $\oklim_{Y_\bullet}(D)$ in detail, we introduce the following alternative construction, which also gives the same convex body as $\oklim_{Y_\bullet}(D)$ after all (see Proposition \ref{prop-oklim Y=V}).
This construction is often more convenient to study the asymptotic invariants of pseudoeffective divisors. To begin with, we first consider a pseudoeffective Cartier $\Z$-divisor $D$ on a variety $X$ of dimension $n$. Let $V \subseteq X$ be a $v$-dimensional irreducible subvariety of $D$.
We will mainly concern the case that $V$ is a positive volume subvariety of $D$.
Fix an admissible flag $V_\bullet$ on $V$:
$$
V_{\bullet}: V=V_0 \supseteq V_1\supseteq\cdots\supseteq V_{v}=\{x\}.
$$
Let $A$ be an ample Cartier $\Z$-divisor on $X$.
For each integer $k\geq 1$, consider the restricted graded linear series $W^k_\bullet:=W_\bullet(kD+A|V)$ of $kD+A$ along $V$ given by
$W_m(kD+A|V)=h^0(X|V,m(kD+A))$ for $m\geq 0$.
We define the \emph{restricted limiting Okounkov body} of a Cartier divisor $D$ along $V$ with respect to $V_{\bullet}$ as
$$\oklim_{V_\bullet}(D):=\lim_{k \to \infty} \frac{1}{k}\okbd_{V_\bullet}(W^k_\bullet) \subseteq \R^{v}_{\geq 0}.$$
One can easily check that $\oklim_{V_\bullet}(D) = \lim_{\eps\to 0+} \okbd_{V_\bullet}(D+\eps A)$.
We can extend this definition for any pseudoeffective $\R$-divisor $D$.

\begin{definition}\label{def-oklim}
Let $D$ be a pseudoeffective divisor on a variety $X$, and $V\subseteq X$ be a $v$-dimensional irreducible subvariety.
The \emph{restricted limiting Okounkov body} $\oklim_{V_\bullet}(D)$ of $D$ along $V$ with respect to $V_\bullet$ as
 $$\oklim_{V_\bullet}(D):=\lim_{\eps\to 0+} \okbd_{V_\bullet}(D+\eps A) \subseteq \R^{v}_{\geq 0}$$
for some ample divisor $A$ on $X$.
If $D$ is not pseudoeffective, then we define $\oklim_{V_\bullet}(D):=\emptyset$.
\end{definition}

This definition depends on the choice of $V$ and the admissible flag $V_\bullet$ on it, but it is independent of the choice of $A$.
By definition, $\oklim_{V_\bullet}(D)$ is a closed convex subset of $\R^{v}$.
By the inclusion $\R^{v} = \{0\}^{n-v} \times \R^{v}  \hookrightarrow \R^n$,
we often treat $\oklim_{V_\bullet}(D)$ as a subset of $\R^n$.


The following is the main property of the restricted limiting Okounkov body $\oklim_{V_\bullet}(D)$.

\begin{theorem}\label{oklimv}
Let $D$ be a pseudoeffective divisor on a variety $X$, and $V$ be a $v$-dimensional irreducible subvariety of $X$. Fix an admissible flag $V_\bullet$ on $V$.
Suppose that $V \not\subseteq \bm(D)$ and $\vol_{X|V}^+(D) > 0$.
Then we have
$$
\dim \oklim_{V_\bullet}(D)=v \text{\;\; and \;\;} \vol_{\R^{v}}(\oklim_{V_\bullet}(D))=\frac{1}{v!}\vol_{X|V}^+(D).
$$
\end{theorem}

\begin{proof}
We first consider the case where $D$ is a $\Z$-divisor. Let $A$ be an ample $\Z$-divisor on $X$.
Since $V \not\subseteq \B_-(D)$ by definition, it follows that $V \not\subseteq \B_+(kD+A)$ for any $k \geq 1$.
Let $W^k_\bullet:=W_\bullet(kD+A|V)$ be the restricted graded series of $kD+A$ along $V$ for $k \geq 1$.
Then $W^k_\bullet$ satisfies Condition (C) by \cite[Lemma 2.16]{lm-nobody}.
It follows from Theorem \ref{thrm-okbd of big} that for each $k\geq 1$, we have
$$\dim \okbd_{V_\bullet}(W^k_\bullet) =\dim V=v
\text{ and }
\vol_{\R^{v}} (\okbd_{V_\bullet}(W^k_\bullet) )=\frac{1}{v!}\vol_{X|V} (kD+A).
$$
The first equality implies $\dim \oklim_{V_\bullet}(D)=v$, and the second equality yields the following computation:
$$
\begin{array}{rl}
\vol_{\R^{v}}(\oklim_{V_\bullet}(D))&=\vol_{\R^{v}}\left(\lim\limits_{k \to \infty}\frac{1}{k} \Delta_{V_\bullet}(W^k_{\bullet})\right)\\
&=\lim\limits_{k \to \infty}\frac{1}{k^{v}}\vol_{\R^{v}}( \Delta_{V_\bullet}(W^k_{\bullet}))\\
&=\lim\limits_{k \to \infty} \frac{1}{v!}\vol_{X|V}\left(D + \frac{1}{k}A\right)\\
&= \frac{1}{v!}\vol_{X|V}^+(D).
 \end{array}
$$
Thus we have shown the assertion for $\Z$-divisors. When $D$ is a $\Q$-divisor, there is a sufficiently divisible integer $m>0$ such that $mD$ is a $\Z$-divisor. It is easy to check that $\oklim_{V_\bullet}(D) = \frac{1}{m} \oklim_{V_\bullet}(mD)$. Thus the assertions for $\Q$-divisors follow from the $\Z$-divisor case.

Now, we assume that $D$ is an $\R$-divisor. Then there exists a sequence of ample divisors $A_i$ on $X$ such that $D+A_i$ are $\Q$-divisor and $\lim_{i \to \infty} A_i = 0$. Note that $V \not\subseteq \bp(D+A_i)$. Thus $V \not\subseteq \bm(D+A_i)$ and $\vol_{X|V}^+(D + A_i) = \vol_{X|V}(D + A_i) > 0$. Since $\oklim_{V_\bullet}(D) = \lim_{i \to \infty} \oklim_{V_\bullet}(D+ A_i)$, the assertion for $\R$-divisors follow from the $\Q$-divisor case.
\end{proof}

\begin{proposition}\label{prop-oklim Y=V}
Let $D$ be a pseudoeffective divisor on a  variety $X$ of dimension $n$, and fix an admissible flag $Y_\bullet$ containing a positive volume subvariety $V=Y_{n-\kappanu(D)}$ of $D$. Consider the $(n-\kappanu(D))$-th partial flag $V_\bullet := Y_{n-\kappanu(D)\bullet}$. Then
$$\oklim_{Y_\bullet}(D)=\oklim_{V_\bullet}(D).$$
In particular, $\oklim_{Y_\bullet}(D)$ depends on the $(n-\kappanu(D))$-th partial flag $V_\bullet$.
\end{proposition}

\begin{proof}
We will denote $\kappanu:=\kappanu(D)$. If $\kappanu=n$, then there is nothing to prove.
Therefore, we assume that $0\leq\kappanu<n$.
Fix an ample divisor $A$ on $X$.
By definition, the positive volume subvariety $Y_{n-\kappanu}=V$ is not contained in $\bm(D)$.
Since $Y_i\supseteq Y_{n-\kappanu}$ for $0\leq i\leq n-\kappanu$, we have $Y_i\not\subseteq\bp(D+\eps A)$ for any $\eps>0$.
Thus Theorem \ref{thrm-okbd slice} implies that if $0\leq i\leq n-\kappanu$, then
$\okbd_{Y_\bullet}(D+\eps A) \cap (\{0\}^{i} \times \R^{n-i}) = \okbd_{Y_{i\bullet}}(D + \eps A)$.
By taking the limit $\eps\to 0$,  for $0\leq i\leq n-\kappanu$, we obtain
\begin{equation}\tag{*}\label{*}
\oklim_{Y_\bullet}(D) \cap (\{0\}^{i} \times \R^{n-i}) = \lim_{\epsilon \to 0+} \okbd_{Y_{i\bullet}}(D+\epsilon A) =: \oklim_{Y_{i\bullet}}(D).
\end{equation}
Since $Y_{n-\kappanu}=V$, it is enough to check
$\oklim_{Y_\bullet}(D)\subseteq\{0\}^{n-\kappanu}\times\R^{\kappanu}$ to prove $\oklim_{Y_\bullet}(D)=\oklim_{V_\bullet}(D)$.
Suppose that this does not hold, that is,
$$\oklim_{Y_\bullet}(D)\not\subseteq\{0\}^{n-\kappanu}\times\R^{\kappanu}\subseteq\R^n.$$
We will derive a contradiction to \cite[Theorem 1.1 (2)]{lehmann-nu}, which says that $\kappanu$ is the maximum of integers $k \geq 0$ such that there exists a constant $C>0$ with $\vol_X(D+ \eps A) \geq C \eps^{n-k}$ for all $\eps >0$.
By Theorem \ref{oklimv}, $l:=\dim\oklim_{Y_\bullet}(D)>\kappanu=\dim\oklim_{V_\bullet}(D)$ and we have
$$
\oklim_{Y_\bullet}(D)\subseteq\R^{i_1}\times\R^{i_2}\times\cdots\times\R^{i_{n-\kappanu}}\times\R^{\kappanu}
$$
where $i_j=0$ or $1$, $\sum_{j=1}^{n-\kappanu}i_{j}=l-\kappanu$ and $\R^0:=\{0\}$.
For simplicity, we may just assume
$$
\oklim_{Y_\bullet}(D)\subseteq\R^{l-\kappanu}\times\{0\}^{n-l}\times\R^{\kappanu}.
$$
Now consider a sufficiently positive ample divisor $A'$ so that $\okbd_{Y_\bullet}(\epsilon A')$ contains $$
\triangle_{\epsilon} := \left\{(x_1, \cdots, x_n) \left| x_i \geq 0 \text{ and } \sum_{i=1}^{n} x_i \leq \epsilon \right.\right\}
$$
for all $\epsilon >0$. We may assume that $A$ is sufficiently ample so that $A'':=A-A'$ is also ample.
By the convexity of $\okbd_{Y_\bullet}(D)$ (cf. \cite[Proof of Corollary 4.12]{lm-nobody}), we have
$$
\oklim_{Y_\bullet}(D)+\triangle_{\epsilon}\subseteq\oklim_{Y_\bullet}(D) + \okbd_{Y_\bullet}(\epsilon A')  \subseteq  \okbd_{Y_\bullet}(D+\epsilon A'') + \okbd_{Y_\bullet}(\epsilon A')  \subseteq \okbd_{Y_\bullet}(D+\epsilon A).
$$
Since $\dim \oklim_{Y_\bullet}(D) = l$, it follows that
$$
C\cdot\vol_{\R^l}\oklim_{Y_\bullet}(D)\cdot \epsilon^{n-l}\leq \vol_{\R^n}\okbd_{Y_\bullet}(D+\epsilon A)=\frac{1}{n!}\cdot \vol_X(D+\epsilon A).
$$
for some constant $C>0$ depending only on $n$ and $l$. We get a contradiction to \cite[Theorem 1.1 (2)]{lehmann-nu} since $l>\kappanu$.
\end{proof}

\begin{corollary}\label{cor-oklimY}
Let $D$ be a pseudoeffective divisor on a smooth projective variety $X$.
Fix a positive volume subvariety $V\subseteq X$ of $D$ and consider an admissible flag $Y_\bullet$ on $X$ containing $V$. Then $\oklim_{Y_\bullet}(D) \subseteq \{0 \}^{n-\kappa_{\nu}(D)} \times \R^{\kappa_{\nu}(D)}_{\geq 0}$.
Furthermore, we have
$$
\dim \oklim_{Y_\bullet} (D)=\kappa_{\nu}(D) \text{\;\; and\;\; } \vol_{\R^{\kappa_{\nu}(D)}}(\oklim_{Y_\bullet} (D))=\frac{1}{\kappa_{\nu}(D)!} \vol^+_{X|V}(D)
$$
where $\oklim_{Y_\bullet}(D)$ is regarded as a convex subset of $\R^{\kappa_{\nu}(D)}$.
\end{corollary}
\begin{proof}
Immediate by Theorem \ref{oklimv} and Proposition \ref{prop-oklim Y=V}.
\end{proof}

\begin{remark}
We do not need to assume the generality of $Y_n=\{x\}$ as we did for the valuative Okounkov body in Theorem \ref{okval-main}.
\end{remark}

\begin{remark}\label{remk-boucksom}
Boucksom also studies the limiting Okounkov body (the numerical Okounkov body in his terminology) of a pseudoeffective divisor with respect to any admissible flag in \cite[4.1.3]{B}. He defines the limiting Okounkov body as a fiber of the global Okounkov body (see \cite[Theorem B]{lm-nobody}) over a given pseudoeffective divisor class. He proves that $\dim \oklim_{Y_\bullet}(D) \leq \kappanu(D)$ for any admissible flag $Y_\bullet$ (\cite[Lemma 4.8]{B}) and the inequality can be strict (\cite[Example 4.14]{B}).
Furthermore, he also shows that there exists an admissible flag $Y_{\bullet}$ such that $\dim \oklim_{Y_\bullet}(D) = \kappanu(D)$ when $D$ is nef (\cite[Proposition 4.9]{B}). Our result Corollary \ref{cor-oklimY} generalizes his result.
\end{remark}

Using Corollary \ref{cor-oklimY}, we can give a new characterization of numerical Iitaka dimension.

\begin{corollary}
Let $D$ be a pseudoeffective divisor on a smooth projective variety $X$.
Then we have
$$
\kappanu(D)=\max\{\dim \oklim_{Y_\bullet}(D) \mid \text{ $Y_\bullet$ is an admissible flag on $X$ }\}.
$$
\end{corollary}

\begin{proof}
By \cite[Lemma 4.8]{B}, we know that $\dim \oklim_{Y_\bullet}(D) \leq \kappanu(D)$ for any admissible flag $Y_\bullet$ on $X$. If $Y_\bullet$ contains a positive volume subvariety, then it follows from Corollary \ref{cor-oklimY} that $\dim \oklim_{Y_\bullet}(D) = \kappanu(D)$. Thus the assertion follows.
\end{proof}

If $D$ is a big divisor, then by \cite[Proposition 4.1]{lm-nobody} and \cite[Theorem A]{Jow} it is known that $\okbd_{Y_\bullet}(D)=\okbd_{Y_\bullet}(D')$ for all admissible flags $Y_\bullet$ on $X$
if and only if $D\equiv D'$. Remark \ref{ex-okval-not num} implies that the statement is false in general in the pseudoeffective case for $\okval_{Y_\bullet}(D)$.
We prove next that $\oklim_{Y_\bullet}(D)$ is an appropriate generalization of $\okbd_{Y_\bullet}(D)$ which makes the statement true.

For a pseudoeffective divisor $D$, we define the following set of divisors
$$
div(\oklim(D)):=\left\{E\left |
\begin{array}{l}
\text{ there exists an admissible flag $Y_\bullet$ with $Y_1=E$}\\
 \text{ such that} \oklim_{Y_\bullet}(D)_{x_1=0}=\emptyset
\end{array} \right.\right\}.
$$

\begin{lemma}\label{lem-div oklim}
The divisors in $div(\oklim(D))$ are precisely the divisorial components of $\bm(D)$ and
the set $div(\oklim(D))$ is finite.
\end{lemma}

\begin{proof}
Let $Y_\bullet$ be an admissible flag on $X$ such that $Y_1=E$.
By the definition of the Okounkov body, it is easy to see that for any ample divisor $A$ and $\eps>0$,
$$
\ord_E(||D+\eps A||)=\inf\left\{x_1\left|(x_1,\cdots,x_n)\in\okbd_{Y_\bullet}(D+\eps A)\right.\right\}.
$$
If we let $\eps\to 0$, then we have
$$
\ord_E(||D||)=\inf\left\{x_1\left|(x_1,\cdots,x_n)\in\oklim_{Y_\bullet}(D)\right.\right\}.
$$
By Theorem \ref{thrm-B- by sigmanum}, $\ord_E(||D||)>0$ if and only if $E\subseteq\bm(D)$.
Thus we obtain the desired statement.
The set $div(\oklim(D))$ is finite by \cite[Corollary 1.11]{nakayama}.
\end{proof}

We now prove a generalization of Jow's theorem for the limiting Okounkov bodies.

\begin{theorem}\label{Jow for oklim}
Let $D$ and $D'$ be pseudoeffective divisors on $X$.
Then the following are equivalent:
\begin{enumerate}
\item[$(1)$] $D \equiv D'$

\item[$(2)$] $\oklim_{Y_\bullet}(D) = \oklim_{Y_\bullet}(D')$ with respect to any admissible flags $Y_\bullet$ on $X$.
\end{enumerate}
\end{theorem}

\begin{proof}
The equivalence $(1)\Leftrightarrow (2)$ is known for big divisors by \cite[Proposition 4.1]{lm-nobody} and \cite[Theorem A]{Jow}.
Thus we assume that $D, D'$ are not big.
The implication (1)$\Rightarrow$(2) follows from the big case:
since $\okbd_{Y_\bullet}(D+\epsilon A)=\okbd_{Y_\bullet}(D'+\epsilon A)$ for any ample divisor $A$ and any $\epsilon>0$, we have
$$
\oklim_{Y_\bullet}(D)=\lim_{\epsilon\to 0}\okbd_{Y_\bullet}(D+\epsilon A)=\lim_{\epsilon\to 0}\okbd_{Y_\bullet}(D'+\epsilon A)=\oklim_{Y_\bullet}(D').
$$
To prove (1)$\Leftarrow$(2), we use the argument used to prove the big case by Jow in \cite{Jow}.
Let $E_1,\cdots, E_l$ be the divisorial components of $\bm(D)$ and $A$ an ample divisor on $X$.
For any sufficiently general admissible flag $Y_\bullet$ on $X$ such that
$Y_{n-1}\not\subseteq\bm(D)$, we have $Y_{n-1}\not\subseteq\bp(D+A)$ for any ample divisor $A$.
We see that
$$
\begin{array}{rl}
&\vol_{\R^1}(\oklim_{Y_\bullet}(D)_{x_1=\cdots=x_{n-1}=0})\\
=& \vol_{\R^1}\left(\lim\limits_{\eps\to 0+}\okbd_{Y_\bullet}(D+\eps A)\right)_{x_1=\cdots=x_{n-1}=0}\\
=& \vol_{\R^1}\left(\lim\limits_{\eps\to 0+}\okbd_{Y_\bullet}(D+\eps A)_{x_1=\cdots=x_{n-1}=0}\right)\\
=&\lim\limits_{\eps\to 0+}\vol_{\R^1}(\okbd_{Y_\bullet}(D+\eps A)_{x_1=\cdots=x_{n-1}=0})\\
=&\lim\limits_{\eps\to 0+}\vol_{X|Y_{n-1}}(D+\eps A)\;\;\;\;\;\;\;(\text{by Theorem \ref{thrm-okbd slice} or \cite[Theorem 3.4 (b)]{Jow}})\\
=&Y_{n-1}\cdot D-\sum\limits_{i=1}^l\sum\limits_{p \in Y_{n-1} \cap E_i }\ord_{E_i}(||D||)\;\;\;\;\;\;(\text{by Lemma \ref{jow lemma}}).
\end{array}
$$
As can be seen in the proof of Lemma \ref{lem-div oklim}, we can read off $\ord_{E_i}(||D||)$ from the limiting Okounkov bodies  $\oklim_{Y'_\bullet}(D)$ with respect to admissible flags $Y'_\bullet$ such that $Y'_1=E_i$.
Therefore, if $\oklim_{Y_\bullet}(D)=\oklim_{Y_\bullet}(D')$ for any admissible flags $Y_{\bullet}$ on $X$, then $Y_{n-1}\cdot D=Y_{n-1}\cdot D'$.
Let $\rho:=\dim\N^1(X)_\R$.
As in \cite[Proof of Theorem A]{Jow}, we can choose general admissible flag $Y^1_\bullet,\cdots,Y^{\rho}_\bullet$ consisting of subvarieties that are transversal complete intersections of very ample divisors such that $Y^1_{n-1},\cdots,Y^{\rho}_{n-1}$ form a basis of $\text{N}_1(X)_\R$.
Since we can read off the intersection numbers $Y^i_{n-1} \cdot D = Y^i_{n-1} \cdot D'$ from the limiting Okounkov bodies $\oklim_{Y_\bullet^i}(D)=\oklim_{Y_\bullet^i}(D')$  for $1 \leq i \leq \rho$, we can conclude that $D\equiv D'$.
\end{proof}

It remains to prove the following lemma that is used in the proof above.

\begin{lemma}[cf. {\cite[Corollary 3.3]{Jow}}]\label{jow lemma}
Let $Y_\bullet$ be a  sufficiently general admissible flag on $X$.
For a pseudoeffective divisor $D$ on $X$, let $E_1,E_2,\cdots,E_l$ be the divisorial components of $\bm(D)$. Then we have
$$
\vol_{X|Y_{n-1}}^+(D)=\lim\limits_{\eps\to 0+}\vol_{X|Y_{n-1}}(D+\eps A)=Y_{n-1}\cdot D-\sum_{i=1}^l\sum_{p \in Y_{n-1} \cap E_i }\ord_{E_i}(||D||).
$$
\end{lemma}

\begin{proof}
Suppose first that $D$ is big.
Since $Y_{n-1}$ is very general, we have $D\in\Bigdiv^{Y_{n-1}}(X)$.
Thus $\vol^+_{X|Y_{n-1}}(D)=\vol_{X|Y_{n-1}}(D)$ and the statement is nothing but \cite[Corollary 3.3]{Jow}.

Now let $D$ be a pseudoeffective divisor.
Applying the statement for the case of big divisors to $\vol_{X|Y_{n-1}}(D+\eps A)$,
we obtain
$$
\begin{array}{rl}
\vol_{X|Y_{n-1}}^+(D)=&\lim\limits_{\eps\to 0+}\vol_{X|Y_{n-1}}(D+\eps A)\\
=&\lim\limits_{\eps\to 0+}\left(Y_{n-1}\cdot (D+\eps A)-\sum\limits_{i=1}^l\sum\limits_{p \in Y_{n-1} \cap E_i }\ord_{E_i}(||D+\eps A||)\right)\\
=&Y_{n-1}\cdot D-\sum\limits_{i=1}^l\sum\limits_{p \in Y_{n-1} \cap E_i }\ord_{E_i}(||D||).
\end{array}
$$
\end{proof}

\begin{remark}\label{remk-fulger0}
Fulger pointed out to us that in the equation of Lemma \ref{jow lemma} we actually have
$$
\sum_{i=1}^l\sum_{p \in Y_{n-1} \cap E_i }\ord_{E_i}(||D||)=Y_{n-1}\cdot N
$$
where $N$ is the negative part of the divisorial Zariski decomposition of $D$.
\end{remark}

\begin{remark}\label{remk-fulger}
In Theorem \ref{Jow for oklim}, it is tempting to expect that $D \equiv D'$ if and only if $\oklim_{Y_\bullet}(D)=\oklim_{Y_\bullet}(D')$ for any admissible flags $Y_{\bullet}$ containing positive volume subvarieties of both $D$ and $D'$.
However, it is not true and we also need to consider the admissible flags not containing the positive volume loci of the divisor.
For example, consider a pseudoeffective divisor $D\not\equiv 0$ on a variety $X$ with $\kappanu(D)=0$. Let $D':=2D$ so that $D\not\equiv D'$. Every point $x \in X$ with $x \not\in \bm(D)=\bm(D')$ is a positive volume subvariety of both $D$ and $D'$. However, we see that for any admissible flag $Y_{\bullet}$ containing $x$, we have
$$
\oklim_{Y_\bullet}(D)=\oklim_{Y_\bullet}(D')=\{ \text{the origin of } \R^{\dim X}\}.
$$
In fact, if we consider the admissible flags $Y_{\bullet}$ containing positive volume subvarieties of both $D$ and $D'$ in Theorem \ref{Jow for oklim}, then using Remark \ref{remk-fulger0} we only obtain $P \equiv P'$ where $D=P+N$ and $D'=P'+N'$ are the divisorial Zariski decompositions.
\end{remark}

\begin{remark}\label{remk-restricted ok}
Let $D$ be a pseudoeffective divisor on $X$, and $V\subseteq X$ be a smooth positive volume subvariety of $D$. Fix an admissible flag $Y_\bullet$ containing $V$. The $(n-\kappanu(D))$-th partial flag $Y_{(n-\kappanu(D))\bullet}$ induces an admissible flag $V_\bullet$ on $V$.
Recall that $\oklim_{Y_\bullet}(D)=\oklim_{V_\bullet}(D)=\lim_{\eps \to 0+}\okbd_{V_\bullet}(D+\eps A)$ where $A$ is an ample divisor on $X$.
By Theorem \ref{D on +vol big}, $D|_V$ is a big divisor on $V$, and $A|_V$ is an ample divisor on $V$.
By definition, $\okbd_{V_\bullet}(D+\eps A) \subseteq \okbd_{V_\bullet}((D+\eps A)|_V)$. Thus we obtain
$$
\oklim_{Y_\bullet}(D) = \lim_{\eps \to 0+}\okbd_{V_\bullet}(D+\eps A) \subseteq  \lim_{\eps \to 0+}\okbd_{V_\bullet}((D+\eps A)|_V) = \okbd_{V_\bullet}(D).
$$
By Theorem \ref{thrm-okbd of big}, we have
$$
\dim\okbd_{V_\bullet}(D|_V)=\kappanu(D) \text{ and }
\vol_{\R^{\kappanu(D)}}\okbd_{V_\bullet}(D|_V)=\frac{1}{\kappanu(D)!}\vol_V(D|_V).
$$
Let $Y_\bullet$ be an admissible flag on $X$ that is general enough so that it contains both a Nakayama subvariety $U$ and a positive volume subvariety $V$ of a pseudoeffective divisor $D$.
Then we have the following inclusions:
\begin{equation}\tag{\#}\label{sharp}
\okval_{Y_\bullet}(D)\subseteq \oklim_{Y_\bullet}(D) \subseteq\okbd_{V_\bullet}(D|_V).
\end{equation}
This confirms the inequalities $\vol_{X|V}(D) \leq \vol_{X|V}^+(D) \leq \vol_{V}(D|_V)$
which we saw in Subsection \ref{subsec-Iitaka}.
If $D$ is big, then all the inclusions are equalities.
However, we will see in the next section that if $D$ is not big, then they are strict in general.
\end{remark}

\section{Examples}\label{exsec}

In this section, we exhibit various examples and counterexamples related to our results.
We start with an example that shows that for some badly chosen admissible flags, Theorems \ref{okval-main} and \ref{oklimv} do not hold.

\begin{example}\label{ex-ok w/ bad flag}
Let $\pi : S \to \P^2$ be a blow-up of $\P^2$ at two distinct points with exceptional divisors $E_1$ and $E_2$, and $L \in |\pi^* \mathcal{O}_{\P^2}(1)|$ be an effective divisor containing $E_1$ but not containing $E_2$. Consider a non-big effective divisor $D:=(L-E_1)+E_2$, and fix an admissible flag
$$
Y_\bullet: S\supseteq E_2 \supseteq \{ x \}
$$
where $x$ is a general point on $E_2$.
Here we note that $E_2$ is neither a Nakayama subvariety nor a positive volume subvariety of $D$.
We have
$$
\okval_{Y_\bullet}(D)=\oklim_{Y_\bullet}(D)=\{(x_1, x_2 ) \in \R^2 \mid 1 \leq x_1 \leq 2 \text{ and } x_2=0 \}
$$
whose Euclidean volume in the $x_1$-axis is 1.
However, note that $\bm(D)=\bp(D+\eps A)=E_2$ for an ample divisor $A$ and a small $\eps >0$. Thus we  see that $\vol_{X|E_2}(D)=\vol_{X|E_2}^+(D)=0$.
We also remark that $\okval_{Y_\bullet}(D)=\oklim_{Y_\bullet}(D)$ is not in the $x_2$-axis as expected.
\end{example}

Examples \ref{okval notsubset oklim} and \ref{ex-okval<oklim} show that the inequalities in (\ref{sharp}) of Remark \ref{remk-restricted ok} can be strict.

\begin{example}\label{okval notsubset oklim}
In this example, we see that $\okval_{Y_\bullet}(D) \subsetneq \oklim_{Y_\bullet}(D)$ even if $\kappa(D)=\kappanu(D)$. Let $S$ be a relatively minimal rational elliptic surface, and $H$ a sufficiently positive ample divisor. Take a general element $V \in |H|$. Fix an admissible flag
$$
{Y_\bullet}: S\supseteq V\supseteq \{ x \}.
$$
Note that $\kappa(-K_S)=\kappanu(-K_S)=1$, and $V$ is both a Nakayama subvariety and a positive volume subvariety of $-K_S$. Thus we have
$\vol_{\R^1}(\okval_{Y_\bullet}(-K_S))=\vol_{S|V}(-K_S)$ and $\vol_{\R^1}(\oklim_{Y_\bullet}(-K_S))=\vol_{S|V}^+(-K_S)$. However, we saw in Example \ref{vol<vol+} that $\vol_{S|V}(-K_S) < \vol_{S|V}^+(-K_S)$. Thus $\okval_{Y_\bullet}(-K_S) \subsetneq \oklim_{Y_\bullet}(-K_S)$.
\end{example}

\begin{example}\label{ex-okval<oklim}
Let $S:=\P(E)$ where $E$ is a rank two vector bundle on an elliptic curve $C$ such that it is a nontrivial extension of $\mathcal{O}_C$ by $\mathcal{O}_C$, and $H$ be the tautological divisor of $\P(E)$. Then we can easily check that $H$ is nef and $\kappa(H)=0$ but $\kappanu(H)=1$. Let $F$ be a general fiber of the natural ruling $\pi : S \to C$.
Note that any general point in $S$ is a Nakayama subvariety of $H$ and $F$ is a positive volume subvariety of $H$.
Thus take an admissible flag
$$
Y_\bullet: S\supseteq F\supseteq \{ x \}
$$
where $x$ is a general point on a general fiber $F$.
Then we can easily compute the following.

\begin{enumerate}
 \item $\okval_{Y_\bullet}(H)=\{ (0,0) \}$ and $\vol_{S|\{ x \}}(H)=1$ by convention.
 \item $\oklim_{Y_\bullet}(H)=\{ (x_1, x_2) \mid x_1=0 \text{ and } 0 \leq x_2 \leq 1 \}$ and $\vol^+_{S|F}(H)=1$.
 \item $\okbd_{Y_\bullet}(H|_F)=  \oklim_{Y_\bullet}(H)$.
\end{enumerate}

\noindent For $\oklim_{Y_\bullet}(H)$, we see the convergence of $\lim_{\epsilon \to 0+} \Delta_{Y_\bullet}(H+\epsilon A)$ for any ample divisor $A$ in the following picture.

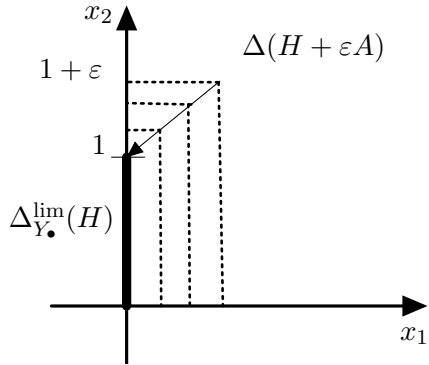
\begin{center}
\begin{tikzpicture}[line cap=round,line join=round,>=triangle 45,x=1.0cm,y=1.0cm]
\clip(-1.58,-0.76) rectangle (4.46,4.3);
\draw [->,line width=1.2pt] (-1.,0.) -- (4.,0.);
\draw [->,line width=1.2pt] (0.,-1.) -- (0.,4.);
\draw (-0.7,4.12) node[anchor=north west] {$x_2$};
\draw (3.5,-0.15) node[anchor=north west] {$x_1$};
\draw (-1.3,3.42) node[anchor=north west] {$1+\epsilon$};
\draw (-0.6,2.4) node[anchor=north west] {$1$};
\draw (-0.2,1.98)-- (0.24,1.98);
\draw [line width=3.5pt] (0.,0.)-- (0.,1.98);
\draw (-1.7,1.52) node[anchor=north west] {$\oklim_{Y_\bullet}(H)$};
\draw (1.4,3.76) node[anchor=north west] {$\Delta(H+\epsilon A)$};
\draw [->] (1.239945618805368,2.9917884396105605) -- (0.,1.98);
\draw [dotted, line width = 1.0pt] (1.2221979352969778,2.97730643271339)-- (1.28,0.);
\draw [dotted, line width = 1.0pt] (1.2221979352969778,2.97730643271339)-- (0.,2.98);
\draw [dotted, line width = 1.0pt] (0.8591403270475781,2.6810535282567285)-- (0.,2.7);
\draw [dotted, line width = 1.0pt] (0.8196014056136802,2.6819254731929894)-- (0.84,0.);
\draw [dotted, line width = 1.0pt] (0.,2.34)-- (0.44268064962666415,2.341224844814632);
\draw [dotted, line width = 1.0pt] (0.44268064962666415,2.341224844814632)-- (0.46,0.);
\end{tikzpicture}
\end{center}
\end{example}

In fact, we use the following theorem to compute the limiting Okounkov body $\oklim_{Y_\bullet}(D)$ of a pseudoeffective divisor $D$ on a surface (cf. \cite[Theorem 6.4]{lm-nobody})

\begin{theorem}\label{surfokbd}
Let $S$ be a smooth projective surface, and $D$ be a pseudoeffective divisor on $S$. Fix an admissible flag $Y_\bullet: S\supseteq C \supseteq \{ x \}$.
Let $a:=\mult_C N$ where $D=P+N$ is the Zariski decomposition, and $\mu := \sup \{ s \geq 0 \mid D-sC \text{ is pseudoeffective} \}$. Consider a divisor $D_t:= D-tC$ for $a \leq t \leq \mu$. Denote by $D_t=P_t+N_t$ the Zariski decomposition. Let $\alpha(t):=\ord_x(N_t|_C)$ and $\beta(t):=\alpha(t)+C.P_t$. Then the limiting Okounkov body of $D$ is given by
$$
\oklim_{Y_\bullet}(D)=\{(x_1, x_2) \in \R^2 \mid a \leq x_1 \leq \mu \text{ and } \alpha(x_1) \leq x_2 \leq \beta(x_2) \}.
$$
\end{theorem}

\begin{proof}
Let $D^{\epsilon} :=D+\epsilon A$ for some ample divisor $A$ and a positive number $\eps \geq 0$, and $D_t^\epsilon:=D^{\epsilon}-tC$. Denote by $D_t^{\epsilon}=P_t^{\epsilon} + N_t^{\epsilon}$ the Zariski decomposition. By Lemma \ref{zdconti} below, we see that the Zariski decomposition $D_t=P_t+N_t$ is given by
$$
P_t=P_t^0=\lim_{\epsilon \to 0+} P_t^{\epsilon}~~\text{and}~~N_t=N_t^0=\lim_{\epsilon \to 0+} N_t^{\epsilon}.
$$
Since $\oklim_{Y_\bullet}(D)=\lim_{\epsilon \to 0+} \Delta_{Y_\bullet}(D+\epsilon A)$, the assertion now follows from \cite[Theorem 6.4]{lm-nobody}.
\end{proof}

\begin{remark}
In Theorem \ref{surfokbd}, we do not need to assume that an admissible flag $Y_\bullet$ contains a positive volume subvariety.
\end{remark}

In the proof of Theorem \ref{surfokbd}, we use the following continuity property of the Zariski decomposition on a surface.

\begin{lemma}\label{zdconti}
Let $S$ be a smooth projective surface, $D$ be a pseudoeffective divisor and $A$ an ample divisor on $S$. Consider the divisor $D^{\eps}:=D + \eps A$ for $\eps \geq 0$, and denote by $D^{\eps}=P^{\eps} + N^{\eps}$ the Zariski decomposition. Then $P^0 = \lim_{\eps \to 0+} P^{\eps}$ and $N^0=\lim_{\eps \to 0+}N^{\eps}$.
\end{lemma}

\begin{proof}
It is sufficient to show that $N^0=\lim_{\eps \to 0+}N^{\eps}$. We write $N^0=\sum_i a_i^0 D_i$ where $D_i$ are prime divisors and $a_i^0 > 0$. Since $\text{Supp}(N^{\eps})=\bm(D^{\eps}) \subseteq \bm(D^0) = \text{Supp}(N^0)$, we can also write $N^{\eps} = \sum_i a_i^{\eps} D_i$ with $a_i^{\eps} \geq 0$ for any $\eps \geq 0$. Note that $a_i^{\eps} = \ord_{D_i}(||D^{\eps}||)$ for any $\eps \geq 0$ (see \cite[Remark III.1.17 (1)]{nakayama}). By definition, we have
$$
\lim_{\eps \to 0+} \ord_{D_i}(||D^{\eps}||) = \lim_{\eps \to 0+} \ord_{D_i}(||D+\eps A||) = \ord_{D_i}(||D||)=a_i^0.
$$
Thus it follows that $N^0=\lim_{\eps \to 0+}N^{\eps}$.
\end{proof}

Examples \ref{nefex} and \ref{toricex} provide  examples in higher dimensions. It is in general very difficult to compute the Okounkov body in higher dimensions.

\begin{example}\label{nefex}
Let $D$ be a nef divisor on a smooth projective variety $X$, and fix an admissible flag $Y_\bullet$ containing a positive volume subvariety $V=Y_{n-\kappanu(D)}$ of $D$. Then we have
$$
\oklim_{Y_\bullet}(D)=\okbd_{Y_{n-\kappanu(D)}}(D|_V)~~\text{ and } ~~\vol_{\R^{\kappanu(D)}}(\oklim_{Y_\bullet}(D)) = \frac{1}{\kappanu(D)!}(D|_V)^{\kappanu(D)}.
$$
\end{example}

\begin{example}\label{toricex}
Let $X$ be a smooth projective toric variety, and $D$ be a $T$-invariant divisor on $X$.
Fix an admissible flag $Y_\bullet$ consisting of $T$-invariant subvarieties.
In toric geometry, one can associate a rational polytope $P_D$ to $D$ (see \cite{F}).
In \cite[Proposition 6.1]{lm-nobody}, it is shown that $P_D$ is nothing but the Okounkov body $\okbd_{Y_\bullet}(D)$ up to translation when $D$ is big.
When $D$ is pseudoeffective, one can find an admissible flag $Y_\bullet$ consisting of $T$-invariant subvarieties such that
$$
P_D=\okval_{Y_\bullet}(D)=\oklim_{Y_\bullet}(D)
$$
up to translation.
\end{example}

\end{document}